\documentclass{article}
\usepackage{authblk}
\usepackage[utf8]{inputenc}
\usepackage{amsmath}
  \usepackage{paralist}
  \usepackage{graphics} 
\usepackage{color,enumitem,graphicx}
\usepackage[colorlinks=true,urlcolor=blue,citecolor=red,linkcolor=blue,linktocpage,pdfpagelabels,bookmarksnumbered,bookmarksopen]{hyperref} 
  \usepackage{latexsym}
\usepackage{amssymb}
\usepackage{amsthm}
\usepackage{epsfig}
\everymath=\expandafter{\the\everymath\displaystyle}

\usepackage{graphicx}
\usepackage{amssymb}
 \usepackage{graphics} 
  \usepackage{epsfig} 
 \usepackage[colorlinks=true]{hyperref}
\hypersetup{urlcolor=blue, citecolor=red}

\usepackage[numbers,sort&compress]{natbib}

\usepackage{natbib}
\usepackage{graphicx}
\newtheorem{theorem}{Theorem}[section]
\newtheorem{corollary}{Corollary}
\newtheorem{lemma}[theorem]{Lemma}
\newtheorem{proposition}{Proposition}

\newtheorem{example}{Example}

\newtheorem{remark}{Remark}

\usepackage{lineno}

\title{Existence, nonexistence, and asymptotic behavior of solutions for $N$-Laplacian equations involving critical exponential growth in the whole $\mathbb{R}^N$}

\author[1]{Anderson L. A. de Araujo}
\author[2]{Luiz F. O. Faria}

\affil[1]{Universidade Federal de Vi\c{c}osa, Departamento de  Matem\'atica\\ 

Av. Peter Henry Rolfs, s/n, Vi\c{c}osa, MG, Brazil, CEP 36570-900.\\

E-mail: {\tt anderson.araujo@ufv.br}}
\affil[2]{Universidade Federal de Juiz de Fora, ICE, Departamento de Matem\'atica\\

Rua Jos\'e Louren\c{c}o Kelmer, s/n  Juiz de Fora, MG, Brazil, CEP 36036-900\\

E-mail: {\tt luiz.faria@ufjf.edu.br}}

\date{}                     
\begin{document}
\setpagewiselinenumbers

\maketitle
\begin{abstract}
 In this paper, we are interested in studying the existence or non-existence of solutions for a class of elliptic problems involving the $N$-Laplacian operator in the whole space. The nonlinearity considered involves critical Trudinger-Moser growth. Our approach is non-variational, and in this way we can address a wide range of problems not yet contained in the literature. Even $W^{1,N}(\mathbb{R}
^N)\hookrightarrow L^\infty(\mathbb{R}
^N)$ failing, we establish $\|u_{\lambda}\|_{L^\infty(\mathbb{R}^N)} \leq  C
\|u\|_{W^{1,N}(\mathbb{R}^N)}^{\Theta}$ (for some $\Theta>0$), when $u$ is a solution. To conclude, we explore some asymptotic properties.

\end{abstract}
{\bf Keywords:} 
Dirichlet problem; Galerkin approximation; Critical growth; Moser-Trudinger inequality; $N$-Laplacian.\\
{\emph{2020 MSC}}: 35A01; 35A25; 35B09; 35B20; 35B33; 35B40; 35J92.

\section{Introduction}\label{S1}
Let $W_0^{1,p}(\Omega)$, $\Omega\subseteq \mathbb{R}^N$, $N\geq 2$, be the Sobolev space endowed with the norm $$\|u\|_{W^{1,p}(\Omega)}=\left(\int_{\Omega}(|\nabla u|^p+|u|^p)dx\right)^{\frac{1}{p}}.$$
If $p<N$, the critical growth $p^*=Np/(N-p)$ means that $\|u\|_{W_0^{1,p}(\Omega)}\hookrightarrow L^q(\Omega)$, $N\leq p \leq p^*$. The case $p=N$ is a borderline case in the sense of Sobolev embeddings. 
As is well-known, one has $\|u\|_{W_0^{1,N}(\Omega)}\hookrightarrow L^q(\Omega)$, $q\geq N$, but a function in $W_0^{1,N}(\Omega)$ may have a local singularity and this causes the failure of the embedding $W^{1,N}(\Omega)\not\subset  L^{\infty}(\Omega)$. If $\Omega$ is a bounded domain, another kind of maximal growth was proved by Yudovi\v{c} \cite{y},  Poho\v{z}aev \cite{Ph} and Trudinger \cite{t}. They proved, in an independent way, 
 that 
\begin{equation}
u\in W^{1,N}_0(\Omega) \Longrightarrow \int_{\Omega} e^{|u|^{N'}}dx<\infty, 
\end{equation}
where $N'=\frac{N}{N-1}$. Moreover, for any higher growth, the corresponding integral can be infinite for a suitable choice of $u$. After that, Moser \cite{m}  improved this assertion, showing that if $u\in W_0^{1,N}(\Omega)$,  then
\begin{equation}\label{TMz}
    \sup_{\|\nabla u\|_{L^N(\Omega)}\leq 1} \displaystyle \int_{\Omega} e^{\alpha|u|^{N'}}dx\left\{ \begin{array}{ll}
     \leq c|\Omega|,& \mbox{ if }\alpha\leq \alpha_N , \\
     =\infty ,&  \mbox{ if }\alpha > \alpha_N,
\end{array}\right. 
\end{equation}
\noindent where $\alpha_N=N\omega_{N-1}^{1/(N-1)}$, $c$ is a constant which depends on $N$, and $\omega_{N-1}$ is the measure of the unit sphere in $\mathbb{R}^N$. Inequality $\eqref{TMz}$ is now called Moser-Trudinger inequality and the term $e^{\alpha_N|u|^{N'}}$ is known as critical Moser-Trudinger growth.

This well-known Moser-Trundinger inequality has been generalized in many ways. In the case $\Omega=\mathbb{R}^N$, B. Ruf when $N=2$ in \cite{Ruf}, and Y.X. Li and B. Ruf in \cite{LR} for $N>2$, proved the following assertion
\begin{equation}\label{TM-N}
\sup_{\|u\|_{W^{1,N}(\mathbb{R}^N)}\leq 1} \displaystyle \int_{\mathbb{R}^N} \phi_N(\alpha|u|^{\frac{N}{N-1}})dx\left\{ \begin{array}{ll}
\leq C(\alpha,N)^{N'},& \mbox{ if }\alpha\leq \alpha_N ,  \\
=\infty ,&  \mbox{ if }\alpha > \alpha_N,
\end{array}\right. 
\end{equation}
where  $\alpha_N>0$ is a constant and  \begin{equation}\label{phi}
\phi_N(t)=e^t-\sum_{j=0}^{N-2}\frac{t^j}{j!}=\sum_{j=N-1}^{\infty}\frac{t^j}{j!}, \,\, t\geq 0.
\end{equation} 
\begin{remark}\label{R1}
	Notice that $\phi_N(t)$ is a increasing function.
\end{remark}

In this paper, we consider the existence, nonexistence, and asymptotic behavior of positive solutions for a class of $N$-Laplacian elliptic equations related to the critical growth \eqref{TM-N} in the whole space $\mathbb{R}^N$. More precisely, we are concerned with the following problem:  

\begin{equation}\label{P}
    -\Delta_N u+|u|^{N-2}u=\lambda a(x)|u|^{q-2}u+f(u) \,\,\mbox{ in }\,\,
\mathbb{R}^N,
\end{equation}
where $\lambda>
0$, $2\leq N$, $1<q<N$, $a$ is a positive function such that $a\in L^{\frac{N}{N-q}}(\mathbb{R}^N) \cap L^{\infty}(\mathbb{R}^N)$, $f:\mathbb{R}\to
\mathbb{R}$ is a continuous function satisfying the growth
condition 	\begin{equation}\label{growth}
0 \leq	f(t)t\leq a_1|t|^{p}\phi_N(\alpha|t|^{\frac{N}{N-1}}), \,\, t\in\mathbb{R},
\end{equation}
where $N<p<+\infty$, and $\alpha>0$. 

In this setting, we mean by a solution of problem $\eqref{P}$ any function
$u\in W^{1,N}(\mathbb{R}^N)\cap C_{loc}^1(\mathbb{R}^N)$, such that 
\begin{align*}
&\int_{\mathbb{R}^N} \left(|\nabla u|^{N-2}\nabla u
\nabla \phi   + |u|^{N-2}u\phi   
-\lambda a(x)|u|^{q-2}u\phi - f(u)\phi \right)dx=0 ,
\end{align*}
for all $\phi\in C_0^\infty(\mathbb{R}^N)$.

Besides existence, this paper is regarding explore some asymptotic properties of the obtained solutions. If $u_{\lambda}$ is a solution of \eqref{P}, we are interested in investigating the following properties:
\begin{equation}\label{conv11}
u_{\lambda}(x)\to 0 \mbox{ as } |x|\to\infty,
\end{equation}
\begin{equation}\label{conv2}
\|u_{\lambda}\|_{W^{1,N}(\mathbb{R}^N)}\rightarrow 0 \mbox{ as } \lambda \to 0^+,
\end{equation}
and
\begin{equation}\label{conv3}
\|u_{\lambda}\|_{L^{\infty}(\mathbb{R}^N)}\rightarrow 0 \mbox{ as } \lambda \to 0^+.
\end{equation}
Furthermore, concerning to problem \eqref{P}, we scrutinize the nonexistence of solution for $\lambda$ large enough.

The class of equations in \eqref{P} appears as a model for several problems in the fields of electromagnetism, astronomy, and fluid dynamics. They can be used to accurately describe the behavior of electric, gravitational, and fluid potentials. Problems of the form \eqref{P} are important in several applications to the study of evolution equations of $N$-laplacian type that appear in non-Newtonian fluids, turbulent flows in porus media and other contexts.

The case $N=2$ is also related to the stationary problem associated with the following initial value Schr\"odinger equation 
\begin{equation}\label{s1}
i \partial_t u+\Delta u+ \mu  g(u)=0, \;\; \mbox{ in } \mathbb{R}_t\times \mathbb{R}^2_x,
\end{equation}
with data
\begin{equation}\label{s2}
u_0:= u(0,\cdot) \in W^{1,2}(\mathbb{R}^2),
\end{equation}
where $u:= u(t,x)$ is a complex-valued function of $(t,x)\in \mathbb{R}\times \mathbb{R}^2$, and 
\begin{equation}
g(u):=u(e^{4\pi|u|^2}-1).
\end{equation}
Schr\"odinger equations involving exponential nonlinearities have several applications, such as the self-trapped beans in plasma \cite{LLT}. In \cite{C}, the author considered Schr\"odinger equation with decreasing exponential nonlinearity.
The stationary problem associated to \eqref{s1} is given by
\begin{equation}\label{s3}
\Delta \phi-\phi +g(\phi)=0, \;\; \phi\in W^{1,2}(\mathbb{R}^2).
\end{equation}
If $\phi\in W^{1,2}(\mathbb{R}^2)$ is a solution to \eqref{s3}, then $e^{it}\phi$ is a solution to \eqref{s1} called a soliton or a standing wave.  Equations of type \eqref{s3} arise in various other contexts of physics such that,
the classical approximation in statistical mechanics, constructive field theory,
false vacuum in cosmology, nonlinear optics, laser propagations, etc. They are
also called nonlinear Euclidean scalar field equations, see \cite{AD, Co, Fr, Gi}.




Recently, there has been considerable interest in the study of existence results for problems of the form
\begin{equation}\label{Pgeral}
\left\{
\begin{array}{lcc}
-\Delta_N u+V(x)|u|^{N-2}u = g(x,u) &\textup{in}& \mathbb{R}^N ,\\
u>0 &\textup{on}& \mathbb{R}^N,
\end{array}
\right.
\end{equation}
where $g(x,u)$ is continuous and behaves like $\phi_N(\alpha|u|^{\frac{N}{N-1}})$ as $|u| \rightarrow +\infty$, we would like to mention  \cite{Freitas1, doO,MS,Aouaoui,dOMS}. 

In general, the potential $V: \mathbb{R}^N \to \mathbb{R}$ is bounded away from zero, i.e.
\[
V(x)\geq c>0 \,\,\, x \in \mathbb{R}^N.
\]
If $V$ is large at infinity in some suitable sense, then the loss of compactness due to the unboundedness of the domain $\mathbb{R}^{N}$ can be overcome and vanishing phenomena can be ruled out. So, a natural framework for the function space setting of problem \eqref{Pgeral} is given by the subspace $E$ of $W^{1, N}\left(\mathbb{R}^{N}\right)$ defined as
\[
E:=\left\{u \in W^{1, N}\left(\mathbb{R}^{N}\right):\int_{\mathbb{R}^{N}} V(x)| u|^{N} d x<+\infty\right\}
\]
endowed with the norm
\[
\|u\|_{E}:=\left(\int_{\mathbb{R}^{N}}\left(|\nabla u|^{N}+V(x)|u|^{N}\right) d x\right)^{\frac{1}{N}} \forall u \in E.
\]
Under appropriate assumptions on the potential $V$, the embedding
\begin{equation}\label{compact1}
E \hookrightarrow L^{p}\left(\mathbb{R}^{N}\right)
\end{equation}
turns out to be compact. For instance, if
\begin{equation}\label{compact2}
V^{-1} \in L^{\frac{1}{N-1}}\left(\mathbb{R}^{N}\right),
\end{equation}
then the embedding \eqref{compact1} is compact for any $p \geq 1$ (see e.g. \cite[Lemma 2.4]{Y}), while assuming the weaker condition
\begin{equation}\label{compact3}
V^{-1} \in L^{1}\left(\mathbb{R}^{N}\right)
\end{equation}
the embedding \eqref{compact1} becomes compact only for $p \geq N$ (see \cite{Costa},  see also \cite{doOMM} for other compactness results).

The authors of \cite{doO,Y,dOMS,lg0}, considering a potential $V$ satisfying \eqref{compact2} or \eqref{compact3}, obtained existence results for equations of the form \eqref{Pgeral} and even more general equations. Their proofs rely, crucially, on the compact embeddings \eqref{compact1}, given by \eqref{compact2} and \eqref{compact3}  (in particular, on the compact embedding of $E$ into $L^{N}\left(\mathbb{R}^{N}\right)$).

Most papers treat problem \eqref{Pgeral} employing variational methods. Then, usually, it is assumed that $g$ has subcritical or critical growth and sometimes asking 
\begin{equation}\label{subcri}
    g(s) \geq c |s|^{p-1}, \mbox{ for each } s\geq 0 \mbox{ where } c>0 \mbox{ and } p>N 
\end{equation} 
are constants, see \cite{Freitas1,ASM}. 

Another common assumption on $g$ is the so-called Ambrosetti-Rabinowitz condition, that is,
$$
\exists R>0 \mbox{ and } \theta > N \mbox{ such that } 0< \theta G(x,s) \leq s g(x,s) \,\, \forall |s| \geq R \mbox{ and } x \in \mathbb{R}^N,
$$
where $G(s)=\int_0^s g(t)dt$, see \cite{Freitas1,doO,MS,lg0}. 


Even when the Ambrosetti-Rabinowitz is dropped, it is usually to be assumed some additional condition to obtain compactness of the Palais-Samle sequences or Cerami sequences. See, for instance, \cite{ASM,lg0}, where the authors assume, respectively, the set of conditions

\begin{equation}\label{ma1}\begin{array}c
\left.\begin{array}{r}
\exists t_0 > 0 \mbox{ and } M > 0 \mbox{ such that }
0 < G(x,s) \leq M g(x,s),\\ \forall |s| \geq t_0 \mbox{ and } x \in \Omega;\\
    0 < N G(x,s) \leq s g(x,s) \,\, \forall |s| \geq 0 \mbox{ and } x \in \Omega, 
\end{array}\right\}
\end{array}
    \end{equation}
and 
\begin{equation}\label{pseudoAR}
 \left. \begin{array}{r}
 H(x, t) \leqslant H(x, s) \mbox{ for all } 0<t<s, \forall x \in \mathbb{R}^{N}, \\
 \mbox{where } H(x, u)=u f(x, u)-N F(x, u).
 \end{array}\right\} 
\end{equation}

The case when the potential $V$ is constant, i.e. $V(x)=c$ for any $x \in \mathbb{R}^{N}$, is much less developed. In such a case, only a few existence results are known, and we should like to mention  \cite{MS} (and references therein), where the authors proved existence results by means Variational approach.
In \cite{Freitas1}, the author consider $V(x)=1$ and $g(x,t)=\lambda h(x)|t|^{q-2}t+f(t), \, x \in \mathbb{R}^{N}$, as in \eqref{growth}, with $f \in C^1(\mathbb{R},\mathbb{R})$ satisfying the Ambrosetti-Rabinowitz condition.\\

In the following, we state our main results.
\begin{theorem}\label{TP}
	Suppose that $f:\mathbb{R}\to
	\mathbb{R}$ is a continuous function satisfying   assumption $\eqref{growth}$. Then there exists $\lambda^{*}>0$ such that for every $\lambda\in (0,\lambda^{*})$ problem $\eqref{P}$ admits a positive solution
	$u_{\lambda}\in W^{1,N}(\mathbb{R}^N)\cap 	C^1_{loc}(\mathbb{R}^N)$. Furthermore, $$\|u_{\lambda}\|_{W^{1,N}(\mathbb{R}^N)}\rightarrow 0,$$ as $\lambda \rightarrow 0^+$.
\end{theorem}

\begin{proposition} \label{prop-n}
	Suppose that $f(t) = |t|^{p-1}\phi_N(\alpha|t|^{\frac{N}{N-1}})$ and let
	\[
	\lambda^*=\sup\{\lambda>0; \eqref{P} \mbox{ has a solution } u_{\lambda} \mbox{ in } W^{1,N}(\mathbb{R}^N)\}.
	\]
	Then $\lambda^* < \infty$.
\end{proposition}


\begin{theorem}\label{T:infinity}
	Suppose that $f:\mathbb{R}\to
	\mathbb{R}$ is a continuous function satisfying   assumption $\eqref{growth}$.
	Then any
	solution $u_{\lambda}$, $\lambda\in (0,\lambda^{*})$, given by Theorem \ref{TP}, satisfies
	$u_{\lambda}(x)\to 0$ as $|x|\to\infty$.
\end{theorem}

\begin{corollary}\label{c1}
	Any	solution $u_{\lambda}$, $\lambda\in (0,\lambda^{*})$, given by Theorem \ref{TP}, satisfies 
	$$\|u_{\lambda}\|_{L^{\infty}(\mathbb{R}^N)}\rightarrow 0,$$
	as $\lambda \rightarrow 0^+$.
\end{corollary}

Notice that in this paper we don not impose any extra hypotheses on $f$ beyond \eqref{growth}. To compare our condition on the nonlinearity of problem \eqref{P} with we have found in the literature, we present two examples that follow.

\begin{example}
As the first example of a function satisfying \eqref{growth}, we have
\begin{equation}\label{Exam}
    f(t):=h(t)\phi_N(\alpha|t|^{\frac{N}{N-1}})
\end{equation}
where
\[
h(t)=|t|^{p-2}t\sin^2{(t)}.
\]
Now, define
\begin{equation}\label{Exam1}
    g(x,t):=\lambda a(x)t^{q-1}+h(t)\phi_N(\alpha|t|^{\frac{N}{N-1}}),
\end{equation}
where  $a\in L^{\frac{N}{N-q}}(\mathbb{R}^N) \cap L^{\infty}(\mathbb{R}^N)$ is a positive function. Notice that \eqref{Exam1}  satisfies neither \eqref{subcri} nor Ambrosetti-Rabinowitz condition nor \eqref{pseudoAR}. First, let us verify that $g$ does not satisfy \eqref{subcri}. Let $c>0$ be a positive fixed constant. Taking the sequence $t_k=2k\pi$, with $k$ a positive integer,  $q<N$ and $a \in L^{\infty}(\mathbb{R}^N)$, there holds $g(x,t_k)=\lambda a(x)t_k^{q-1}<ct_k^{p-1}$, for $k$ large enough. 

Now, let us define
\[ \begin{array}{lll}
H(x,t)&=&t[\lambda a(x)t^{q-1}+|t|^{p-2}t\sin^2{(t)}\phi_N(\alpha|t|^{\frac{N}{N-1}})]\\
&&-N\int_0^t(\lambda a(x)s^{q-1}+|s|^{p-2}s\sin^2{(s)}\phi_N(\alpha|s|^{\frac{N}{N-1}}))ds.
\end{array}\]
Still considering the sequence $t_k=2k\pi$, with $k$ a positive integer, $H$ satisfies
\[
H(x,t_k)=\lambda a(x)t_k^{q}\left(1-\frac{N}{q}\right)-N\int_0^{t_k}|s|^{p-2}s\sin^2{(s)}\phi_N(\alpha|s|^{\frac{N}{N-1}})ds<0.
\]
Notice that $H(x,t_k)>H(x,t_{k+1})$. Therefore, \eqref{Exam1}  satisfies neither Ambrosetti-Rabinowitz condition nor \eqref{ma1} nor \eqref{pseudoAR}.

\end{example}

\begin{example}
The second example of a function satisfying \eqref{growth} we would like to mention is
\begin{equation}\label{Exam2}
    f(t):=|t|^{p-2}t[(\sin{(t)})_+]\phi_N(\alpha|t|^{\frac{N}{N-1}}), \, \mbox{ where } z_+=\max\{z,0\}.
\end{equation}
Notice that $f$ is a continuous function but it is not derivable. Indeed, since 
$$\begin{array}{lll}
f'_-(\pi)&=&\lim_{l \to 0^-}\frac{f(\pi+l)}{l}=\lim_{l \to 0^-}\frac{|\pi+l|^{p-2}(\pi+l)\sin{(\pi+l)}\phi_N(\alpha|\pi+l|^{\frac{N}{N-1}})}{l}\\&=&-\lim_{l \to 0^-}|\pi+l|^{p-2}(\pi+l)\frac{\sin{(l)}}{l}\phi_N(\alpha|\pi+l|^{\frac{N}{N-1}})\\
&=&-\pi^{p-1}\phi_N(\alpha\pi^{\frac{N}{N-1}})<0,\end{array}$$
and 
$$f'_+(\pi)=\lim_{l \to 0^+}\frac{f(\pi+l)}{l}=0,
$$
so $f'(\pi)$ does not exists. In particular, $f$ does not belong to  $C^1(\mathbb{R},\mathbb{R})$, as it imposed  in \cite{Freitas1} to obtain the existence result. Furthermore, if $N\geq 2$, \eqref{Exam2}  satisfies neither Ambrosetti-Rabinowitz condition nor \eqref{pseudoAR}. Indeed, let us define
\[ \begin{array}{lll}
\tilde{H}(t)&=&t(|t|^{p-2}t[(\sin{(t)})_+]\phi_N(\alpha|t|^{\frac{N}{N-1}}))\\
&&-N\int_0^t|s|^{p-2}s[(\sin{(s)})_+]\phi_N(\alpha|s|^{\frac{N}{N-1}})ds.
\end{array}\]
Considering the sequence $t_k=2k\pi$, with $k$  being a positive integer, $\tilde{H}$ satisfies
\[
\tilde{H}(t_k)=-N\int_0^{t_k}|s|^{p-2}s[(\sin{(s)})_+]\phi_N(\alpha|s|^{\frac{N}{N-1}})ds<0.
\]
Notice that $\tilde{H}(t_k)>\tilde{H}(t_{k+1})$. Therefore, \eqref{Exam2} does not satisfies neither Ambrosetti-Rabinowitz condition  nor \eqref{ma1} nor \eqref{pseudoAR}.
\end{example}

By observing these simple examples above, we can see that our results are not included in the previous literature. Therefore, it brings novelty in the study of such equations in the field.

The solution of \eqref{P} is obtained as the limit of auxiliarily problems in bounded domains. The technique combines the Galerkin method, comparison principle of lower and upper solutions, and regularity scheme. 
As in  \cite{AL}, we also consider a special class of normed spaces of finite dimension. However, to clarify this approach we show Lemma \ref{prop1}, which is a result of independent interest.  
Such a lemma, together Brouwer's fixed point theorem, allows us, in the scheme of the Galerkin method, to work in general Banach spaces of finite dimensions with general norms.

 Due to the presence of the critical term $\phi_N(\alpha|u|^{\frac{N}{N-1}})$ and since we are in the whole space $\mathbb{R}^N$, we had to overcome several difficulties. One of them is a suitable improvement needed in the estimates of the approximating functions of $f$ by comparing it with \cite[Lemma
2.2.]{AL} (see Lemma \ref{lemma2} bellow). 
Another delicate point in our approach is the regularity needed to ensure that the solution, obtained in the limit, does not vanish. 
In the bounded domain, we obtain regularity up to the boundary for the auxiliary problems considering approximate functions in the sense of Strauss \cite{s}. For the domain-wide solution case, a kind of a priori estimate in the sup norm is presented, see Proposition \ref{prop:bdd-2} (and the consequences of that outcome).

Now we proceed to introduce the organization of the rest of the paper. Section \ref{Prel} presents comparison principles, some preliminaries results that are useful through the text and an important result labeled as Lemma \ref{prop1}. Section \ref{s.2} is concerning approximating functions that enable us to obtain regularity up to the boundary to the approximating solutions. In Section
\ref{S3}, we study a class of auxiliary problems in bounded domains.
Section \ref{S10} is devoted to the proof of our main results.


\section{Preliminaries}\label{Prel}
In this section, we state some preliminaries results which will be necessary throughout the paper.

Let $u\in W^{1,N}_0(D)$. In what follows, let us denote by $\tilde{u}$ the canonical  extension of $u$ by $0$ outside $D$, that is, 
	
	\begin{equation}\label{utilde}
\tilde{u}(x)= \left\{ 	    \begin{array}{ll}
u(x)& \mbox{ if }x\in D,\\
0& \mbox{ if }x\in \mathbb{R}^N\setminus D.
	    \end{array} \right.
	\end{equation}
	It is well known that  $u\in W_0^{1,N}(D)$ implies $\tilde{u}\in W^{1,N}(\mathbb{R}^N)$ (see e.g. \cite[Proposition 9.18]{Brezis}).

The next technical lemma was proved in \cite[Lemma 2.3]{AF}.
\begin{lemma}\label{lemf}
    Let $\alpha>0$ and $r>1$. Then, for every $\beta>r$, there exists a constant $C:=C(\beta)$ such that
    $$\left(\phi_N(\alpha |u|^{\frac{N}{N-1}})\right)^{r}\leq C \phi_N(\beta\alpha |u|^{\frac{N}{N-1}}),$$
    where $\phi_N$ is given in \eqref{phi}.
\end{lemma}

A complete proof of the next result can be found in \cite[Lemma 3]{FMT}.
\begin{lemma}\label{le2}
	 Let $1<q<N$.
For any constants $b>0$, the problem
    \begin{equation}\label{19}
    \left\{
    \begin{array}{lll}
       -\Delta_N u+|u|^{N-2}u=b|u|^{q-2}u &\mbox{in}&D,\\
        u>0&\mbox{in}&  D,\\
        u=0 & \mbox{on} & \partial D,
    \end{array} \right.
\end{equation}
admits a solution $u_b\in C^1_0(\overline{D})$ satisfying
$\partial u_b/\partial \nu<0$ on $\partial D$.
\end{lemma}

We conclude this section by presenting a lemma, which is a consequence of Brouwer’s Fixed Point Theorem. However, our statement is a subtle (but very useful) generalization by comparing it with the literature. In particular, this result allows us to work in general Banach spaces, with freedom in choosing the norm  (see the proof of Lemma \ref{lem:appro-sol}). We will adopt $|x|_2=\sqrt{\left\langle x,x\right\rangle}$ to denote the usual euclidean norm in $\mathbb{R}^d$ and $\|x\|_d$ to denote a general norm in $\mathbb{R}^d$. The proof of next lemma follows some ideas like in Kesavan \cite{k}, where the result for the particular case $\|x\|_d:=|x|_2$  is presented.

	\begin{lemma}\label{prop1}
		Let $F: (\mathbb{R}^d, \|\cdot\|_d) \rightarrow (\mathbb{R}^d, \|\cdot\|_d)$ be a continuous function such that $\left\langle F(\xi),\xi\right\rangle\geq 0$ for every $\xi \in \mathbb{R}^d$ with $\|\xi\|_{d}=\varrho$ for some $\varrho>0$,  and $\langle \cdot, \cdot \rangle ^{1/2}=|\cdot|_2$. Then, there exists $z_0$ in the closed ball $\overline{B}^d_\varrho(0):=\{z \in \mathbb{R}^d; \|z\|_{d}\leq \varrho\}$ such that $F(z_0)=0$.
	\end{lemma}
	\begin{proof}
	Firstly, there exists $c(d)>0$ such that
	\begin{equation}\label{equiv.}
	    \|x\|_d\leq c(d)|x|_2, \,\,\ \forall x \in \mathbb{R}^d.
	\end{equation}
	Suppose, $F(x)\neq0$ for all $x \in \overline{B}^d_\varrho(0)$. Define
	\[
	g: (\mathbb{R}^d, \|\cdot\|_d) \rightarrow (\mathbb{R}^d, \|\cdot\|_d)
	\]
	by
	\[
	g(x)=-\frac{\varrho}{\|F(x)\|_d}F(x)
	\]
	which maps $\overline{B}^d_\varrho(0)$ into itself and is continuous. Hence it has a fixed point $x_0$, by Brouwer's Fixed Point Theorem. Since $x_0=g(x_0)$, we have $\|x_0\|_d=\|g(x_0)\|_d=\varrho>0$. But then by \eqref{equiv.}
	\[
	0<\varrho^2=\|x_0\|^2_d\leq c(d)^2|x_0|^2_2=c(d)^2\left\langle x_0,x_0\right\rangle=c(d)^2\left\langle x_0,g(x_0)\right\rangle
	\]
	\[
	=-c(d)^2\frac{\varrho}{\|F(x_0)\|_d}\left\langle x_0,F(x_0)\right\rangle\leq 0,
	\]
	by assumptions, which is a contradiction.
	\end{proof}

\subsection{Comparison principle}\label{S2}
In this ection, we assume that $D$ is a bounded domain in
$\mathbb{R}^N$ with $C^2$ boundary $\partial D$. In the following, we present a couple of comparison principles for a subsolution
and a supersolution of the problem
\begin{equation}\label{1z}
\left\{
\begin{array}{ll}
-\Delta_N u+|u|^{N-2}u=g(u)& \mbox{ in } D,\\
u=0& \mbox{ on } \partial D,
\end{array}
\right.
\end{equation}
where
$g:\mathbb{R}\rightarrow\mathbb{R}$ is a continuous function.

We say that $u_1\in W^{1,N}({D})$ is a subsolution of problem
\eqref{1z} if $u_1\leq 0$ on $\partial D$ and
\begin{eqnarray*}
\int_{D}(|\nabla u_1|^{N-2}\nabla u_1\nabla\varphi+ 
|u_1|^{N-2}u_1\varphi )dx 
\leq\int_{D} g(u_1)\varphi dx
\end{eqnarray*}
for all $\varphi\in W^{1,N}_0({D})$ with $\varphi\geq 0$ in ${D}$
provided the integral $\int_{D} g(u_1)\varphi dx$ exists. We say
that $u_2\in W^{1,N}({D})$ is a supersolution of \eqref{1z} if the
reversed inequalities are satisfied with $u_2$ in place of $u_1$ for
all $\varphi\in W^{1,N}_0({D})$ with $\varphi\geq 0$ in ${D}$.

The next comparison results are particular cases of the ones achieved in \cite[Theorem 3, Theorem 5]{FMT}. 

\begin{proposition}\label{teorsubsup}
Let $g:\mathbb{R}\rightarrow\mathbb{R}$ be a continuous function
such that $g(t)/t^{N-1}$ is decreasing for $t>0$. Assume that $u_1$
and $u_2$ are a positive subsolution and a positive supersolution of
problem \eqref{1z}, respectively. If $u_2(x)>u_1(x)=0$ for all
$x\in\partial D$, $u_i\in C^{1,\alpha}(\overline{D})$ with some
$\alpha\in(0,1)$, $\Delta_N  u_i\in L^{\infty}(D)$, for $i,j=1,2$,
then $u_2\geq u_1$ in $D$.
\end{proposition}

Whenever $u_1$ and $u_2$ satisfy the homogeneous Dirichlet
boundary condition we can state the following result.

\begin{proposition}
    \label{com:Dirichlet}
Let $g:\mathbb{R}\rightarrow\mathbb{R}$ be a continuous function
such that $g(t)/t^{N-1}$ is decreasing for $t>0$. Assume that $u_1,u_2\in C_0^{1,\alpha}(\overline{D})$ with some
$\alpha\in(0,1)$, are a positive subsolution and a positive supersolution of
problem \eqref{1z}, respectively. If  $\Delta_N  u_i\in L^{\infty}(D)$, for $i,j=1,2$,  $u_1/u_2\in L^\infty(D)$ and $u_2/u_1\in L^\infty(D)$,
then  $u_2\geq u_1$ in
${D}$.
\end{proposition}



 
 \section{Approximating functions} \label{s.2}

To prove Theorem \ref{TP}, we approximate $f$ by Lipschitz functions $f_k:\mathbb{R} \to \mathbb{R}$ defined by
\begin{equation}\label{eq1}
f_k(s)=\displaystyle\left\{
	\begin{array}{lcc}
-k[G(-k-\frac{1}{k}) - G(-k)], &\textup{if}& s\leq -k,\\	
-k[G(s-\frac{1}{k}) - G(s)], &\textup{if}& -k\leq s \leq -\frac{1}{k},\\
k^2s[G(-\frac{2}{k}) - G(-\frac{1}{k})], &\textup{if}& -\frac{1}{k}\leq s\leq 0,\\
k^2s[G(\frac{2}{k}) - G(\frac{1}{k})], &\textup{if}& 0\leq s\leq \frac{1}{k},\\	
k[G(s+\frac{1}{k}) - G(s)], &\textup{if}& \frac{1}{k}\leq s \leq k,\\
k[G(k+\frac{1}{k}) - G(k)], &\textup{if}& s\geq k,\\
	\end{array}
		\right.
\end{equation}
where $G(s)=\int_0^sf(\xi)d\xi$.

The following approximation result was proved in \cite{s}.

\begin{lemma}\label{lemma1}
Let $f:\mathbb{R} \to \mathbb{R}$ be a continuous function such that $sf(s)\geq 0$ for every $s \in \mathbb{R}$. Then there exists a sequence $f_k:\mathbb{R} \to \mathbb{R}$ of continuous functions satisfying

(i) $sf_k(s)\geq 0$ for every $s \in \mathbb{R}$;

(ii) $\forall \, k \in \mathbb{N}$ $\exists\, c_k>0$ such that $|f_k(\xi) - f_k(\eta)|\leq c_k|\xi - \eta|$ for every $\xi, \eta \in \mathbb{R}$;
	
(iii) $f_k$ converges uniformly to $f$ in bounded subsets of $\mathbb{R}$.
\end{lemma}

The sequence $f_k$ of the previous lemma has some additional properties presented below. Here, we present a suitable improvement in the estimates of the approximating functions  $f_k$ by comparing it with \cite[Lemma
2.2.]{AL}.
\begin{lemma}\label{lemma2}
Let $f: \mathbb{R} \to \mathbb{R}$ be a continuous function satisfying \eqref{growth} for every $s \in \mathbb{R}$.
Then the sequence $f_k$ of Lemma \ref{lemma1} satisfies

(i) $\forall \, k \in \mathbb{N}$, $0\leq sf_k(s) \leq C_1|s|^p\phi_N(2^{\frac{N}{N-1}}\alpha\,|s|^{\frac{N}{N-1}})$ for every $|s|\geq \frac{1}{k}$;
	
(ii) $\forall \, k \in \mathbb{N}$, $0\leq sf_k(s) \leq C_2\frac{1}{k^{p-2}}|s|^2$ for every $|s|\leq \frac{1}{k}$,

\noindent where $C_1$ and $C_2$ are positive constants independent of $k$.
\end{lemma}
\proof Everywhere in this proof, the constant $a_1$ is the one of \eqref{growth}.

\textit{\underline{First step}}. Suppose that $-k\leq s \leq -\frac{1}{k}$.

By the mean value theorem, there exists $\eta\in (s-\frac{1}{k},s)$ such that
\[
f_k(s)=-k[G(s-\frac{1}{k}) - G(s)]=-kG'(\eta)(s-\frac{1}{k}-s)=f(\eta)
\]
and
\[
sf_k(s)=sf(\eta).
\]
Since $s-\frac{1}{k}<\eta<s<0$ and $f(\eta)<0$, we have $sf(\eta) \leq \eta f(\eta)$. Therefore, by using Remark \ref{R1}, we obtain
\[
\begin{array}{rcl}
sf_k(s) \leq \eta f(\eta) &\leq & a_1|\eta|^{p}\phi_N(\alpha\,|\eta|^{\frac{N}{N-1}})\\
 &\leq & a_1|s-\frac{1}{k}|^{p}\phi_N(\alpha\,|s-\frac{1}{k}|^{\frac{N}{N-1}})\\
&\leq & a_1(|s| +\frac{1}{k})^{p}\phi_N(\alpha\,(|s| +\frac{1}{k})^{\frac{N}{N-1}})\\
&\leq& a_1(2|s|)^{p}\phi_N(\alpha\,(2|s|)^{\frac{N}{N-1}})\\
&=& a_12^{p}|s|^{p}\phi_N(2^{\frac{N}{N-1}}\alpha\,|s|^{\frac{N}{N-1}}).
\end{array}
\]

\textit{\underline{Second step}}. Assume $\frac{1}{k}\leq s \leq k$.

By the mean value theorem, there exists $\eta\in (s,s+\frac{1}{k})$ such that
\[f_k(s)=k[G(s+\frac{1}{k}) - G(s)]=kG'(\eta)(s+\frac{1}{k}-s)=f(\eta)\]
and
\[sf_k(s)=sf(\eta).\]
Since $0<s <\eta <s+\frac{1}{k}$ and $f(\eta)>0$, we have $sf(\eta) \leq \eta f(\eta)$. Therefore,
\[
\begin{array}{rcl}
sf_k(s) \leq \eta f(\eta) &\leq & a_1|\eta|^{p}\phi_N(\alpha\,|\eta|^{\frac{N}{N-1}})\\
 &\leq & a_1|s+\frac{1}{k}|^{p}\phi_N(\alpha\,|s+\frac{1}{k}|^{\frac{N}{N-1}})\\
&\leq& a_1(2|s|)^{p}\phi_N(\alpha\,(2|s|)^{\frac{N}{N-1}})\\
&=& a_12^{p}|s|^{p}\phi_N(2^{\frac{N}{N-1}}\alpha\,|s|^{\frac{N}{N-1}}).
\end{array}
\]

\textit{\underline{Third step}}. Suppose that $|s|\geq k$, then
\begin{equation}\label{eq3}
	f_k(s)=\displaystyle\left\{
	\begin{array}{lcc}
	-k[G(-k-\frac{1}{k}) - G(-k)], &\textup{if}& s\leq -k\\
	k[G(k+\frac{1}{k}) - G(k)], &\textup{if}& s\geq k.\\
		\end{array}
		\right.
\end{equation}

If $s\leq -k$, by the mean value theorem, there exists $\eta\in (-k-\frac{1}{k},-k)$ such that
\[f_k(s)=k[G(-k-\frac{1}{k}) - G(-k)]=-kG'(\eta)(-k-\frac{1}{k}-(-k))=f(\eta)\]
and
\[sf_k(s)=sf(\eta).\]
Since $-k-\frac{1}{k} <\eta <-k<0$ and $k<|\eta| < k + \frac{1}{k}$, we conclude that
\begin{equation}\label{eq3.1}
\begin{array}{rcl}
sf_k(s)=\frac{s}{\eta}\eta f(\eta) &\leq& \frac{|s|}{|\eta|}a_1|\eta|^{p}\phi_N(\alpha\,|\eta|^{\frac{N}{N-1}}) =a_1|s||\eta|^{{p}-1}\phi_N(\alpha\,|\eta|^{\frac{N}{N-1}})\\
 &\leq& a_1|s|(k + \frac{1}{k})^{{p}-1}\phi_N(\alpha\,(k + \frac{1}{k})^{\frac{N}{N-1}})\\
 &\leq& a_1|s|(|s| + \frac{1}{k})^{{p}-1}\phi_N(\alpha\,(|s| + \frac{1}{k})^{\frac{N}{N-1}})\\
&\leq& a_1|s|(2|s|)^{{p}-1}\phi_N(\alpha\,(2|s|)^{\frac{N}{N-1}})\\
&\leq& a_12^{{p}}|s|^{p}\phi_N(2^{\frac{N}{N-1}}\alpha\,|s|^{\frac{N}{N-1}}).
\end{array}
\end{equation}

If $s\geq k$, by the mean value theorem, there exists $\eta\in (k,k + \frac{1}{k})$ such that
\[f_k(s)=k[G(k+\frac{1}{k}) - G(k)]=kG'(\eta)(k+\frac{1}{k}-k)=f(\eta).\]
By computations similar to conclude \eqref{eq3.1} one has
\[sf_k(s)=sf(\eta)=\frac{s}{\eta}\eta f(\eta) \leq \frac{|s|}{|\eta|}a_1|\eta|^{p} \phi_N(\alpha\,|\eta|^{\frac{N}{N-1}}) \leq a_12^{{p}}|s|^{p}\phi_N(2^{\frac{N}{N-1}}\alpha\,|s|^{\frac{N}{N-1}}).\]

\textit{\underline{Fourth step}}. Assume $-\frac{1}{k}\leq s\leq \frac{1}{k}$, then
\begin{equation}
	f_k(s)=\displaystyle\left\{
	\begin{array}{lcc}
	k^2s[G(-\frac{2}{k}) - G(-\frac{1}{k})], &\textup{if}& -\frac{1}{k}\leq s\leq 0\\
	k^2s[G(\frac{2}{k}) - G(\frac{1}{k})], &\textup{if}& 0\leq s\geq \frac{1}{k}.\\
		\end{array}
		\right.
\end{equation}

If $-\frac{1}{k}\leq s\leq 0$, by the mean value theorem, there exists $\eta\in (-\frac{2}{k},-\frac{1}{k})$ such that
\[f_k(s)=k^2s[G(-\frac{2}{k}) - G(-\frac{1}{k})]=k^2sG'(\eta)(-\frac{2}{k}-(-\frac{1}{k}))=-ksf(\eta).\]
Therefore
\begin{equation}\label{eq3.2}
    \begin{array}{lll}
    sf_k(s)&=&-ks^2f(\eta)=-k\frac{s^2}{\eta}\eta f(\eta) \leq k\frac{s^2}{|\eta|}\eta f(\eta) \\

&\leq& a_1k|s|^2|\eta|^{{p}-1}\phi_N(\alpha\,|\eta|^{\frac{N}{N-1}}) \leq a_1k|s|^2(\frac{2}{k})^{{p}-1}\phi_N(\alpha\,|\eta|^{\frac{N}{N-1}})\\ &\leq& a_12^{{p}-1}\frac{|s|^2}{k^{p-2}}\phi\left(\alpha\,\left(\frac{2}{k}\right)^{\frac{N}{N-1}}\right)\\
&\leq& a_12^{{p}-1}\exp(2^{\frac{N}{N-1}}\alpha)\frac{1}{{k^{p-2}}}|s|^2. 
\end{array}
\end{equation}

If $0\leq s \leq \frac{1}{k}$, by the mean value theorem, there exists $\eta\in (\frac{1}{k},\frac{2}{k})$ such that
\[f_k(s)=k^2s[G(\frac{2}{k}) - G(\frac{1}{k})]=k^2sG'(\eta)(\frac{2}{k}-\frac{1}{k})=ksf(\eta).\]
By similar computations to conclude \eqref{eq3.2} one obtains
\[sf_k(s)=ks^2f(\eta)=k\frac{s^2}{|\eta|}\eta f(\eta)  \leq a_12^{{p}-1}\exp(2^{\frac{N}{N-1}}\alpha)\frac{1}{{k^{p-2}}}|s|^2. \]
The proof of the lemma follows by taking $C_1=a_12^{p}$ ad $C_2=a_12^{{p}-1}\exp(2^{\frac{N}{N-1}}\alpha)$, where $a_1$ is given in \eqref{growth}. \qed



\section{Solution on bounded domains}\label{S3}

From now on, we assume that $D$ is a bounded domain in
$\mathbb{R}^N$ with $C^2$ boundary $\partial D$. For $r\geq 1$, we
denote by $\|u\|_{L^r(D)}$ the usual norm on the space $L^r(D)$. We
endow $W_0^{1,N}(D)$ with the norm $\|u\|_{W^{1,N}(D)}^N=\|\nabla
u\|_{L^N(D)}^N+\|u\|_{L^N(D)}^N$.

In this section, we focus on the existence of a positive solution for the problem:

$$\left\{
\begin{array}{lll}
-\Delta_N u+|u|^{N-2}u=\lambda a(x)|u|^{q-2}u+f(u) &\mbox{in}&   D,\nonumber\\
u>0&\mbox{in}&   D,\nonumber\\
u(x)= 0 &\mbox{on}&\partial D.\nonumber
\end{array} \right.\leqno(PD)
$$
Here $\lambda>0$ is a parameter, $1<q<N$, $f: [0, \infty) \rightarrow \mathbb{R}$ is a continuous function satisfying \eqref{growth}.

We say that $u\in W_0^{1,N}(D)$ is a solution of $(PD)$ if $u(x)>0$ in
 $D$ and
\begin{eqnarray*}
&	\displaystyle\int_D |\nabla u|^{N-2}\nabla u\nabla \phi   dx
	+\int_D |u|^{N-2} u\phi dx  
		=\lambda \int_{D}a(x)|u|^{q-2}u\phi  dx  + \int_D f(u)\phi dx,  
\end{eqnarray*}
 for all $\phi\in W_0^{1,N}(D)$.

The existence of a solution for problem $(PD)$ is stated below.

\begin{theorem}\label{TP2}
	Suppose that $f:[0,\infty) \to \mathbb{R}$ is a continuous function satisfying \eqref{growth}. Then there exists $\lambda^*>0$ such that for every $\lambda \in (0,\lambda^*)$  problem
	$(PD)$ admits a solution $u_{\lambda}\in W^{1,N}_0(D)$ such
	that $\partial u_{\lambda}/\partial \nu<0$ on $\partial D$, where $\nu$ stands
	for the outer normal to $\partial D$. 
\end{theorem}


\subsection{Approximate equation} \label{s.3}

In the proof of Theorem~\ref{TP2} we utilize the following
auxiliary problem:
$$\left\{
\begin{array}{lll}
-\Delta_N u+|u|^{N-2}u=\lambda a(x)|u|^{q-2}u+f_n(u)+\frac{\varphi}{n} &\mbox{in}&   D\nonumber\\
u>0&\mbox{in}&   D\nonumber\\
u(x)= 0 &\mbox{on}&\partial D,\nonumber
\end{array} \right.\leqno(PD_n)
$$
with $n>0$ a integer number, $\varphi$ is a fixed positive function such that $\varphi\in L^{\infty}(\mathbb{R}^N)\cap L^{N'}(\mathbb{R}^N)$ and $f_n$  is given by Lemma \ref{lemma1} and Lemma \ref{lemma2}.

\begin{lemma}\label{lem:appro-sol}
	There exists $\lambda^*>0$ and $n^* \in \mathbb{N}$ such that $(PD_n)$ has a  solution
	$u_n\in C^1_0(\overline{D})$ such that $\partial u_n/\partial \nu<0$ on $\partial D$ for every $\lambda \in (0,\lambda^*)$ and $n\geq n^*$. Furthermore, 
		$$\|u_n\|_{W^{1,N}(D)}\leq \varrho,$$
where $\varrho$ does not depend on $n$.
\end{lemma}
\begin{proof}
Let $\mathcal{B}=\{w_1,w_2,\dots,w_n,\dots\}$ be a Schauder basis  (see \cite{FJN,Brezis} for details) for the 
Banach space $(W_0^{1,N}(D),\|\cdot\|_{W^{1,N}(D)})$. For each positive integer $m$, let 
\[
W_m=[w_1,w_2,\dots,w_m]
\]  be the $m$-dimensional subspace of $W^{1,N}_0(D)$ generated by $\{w_1,w_2,\dots,w_m\}$  with norm  induced from  $W^{1,N}_0(D)$.  
Let $ \xi=(\xi_{1},\ldots,\xi_{m})\in \mathbb{R}^m$, notice that 
\begin{equation}\label{norm}
    |\xi|_m := \|\sum _{j=1}^{m} \xi_jw_j\|_{W^{1,N}(D)}
\end{equation}
defines a norm in $\mathbb{R}^m$ (see \cite{AL} for the details).

	By using the above notation, we can identify the spaces $(W_m, \|\cdot \|_{W^{1,N}(D)})$ and $(\mathbb{R}^m,|\cdot|_m)$ by the isometric linear
	transformation
	\begin{equation}\label{transf}
	u=\sum_{j=1}^{m}\xi_{j}w_{j}\in W_m\mapsto
	\xi=(\xi_{1},\ldots,\xi_{m})\in\mathbb{R}^{m}.
	\end{equation}

Now,	define the function $F:\mathbb{R}^m \to \mathbb{R}^m$ such that $$F(\xi)=(F_1(\xi),F_2(\xi),\dots, F_m(\xi)),$$ where $\xi=(\xi_1, \xi_2, ..., \xi_m) \in \mathbb{R}^m$,
	\[\begin{array}{lll}
	F_j(\xi)&=&\int_{D}|\nabla u|^{N-2}\nabla u\nabla w_j dx+ \int_{D}|u|^{N-2}uw_j dx- \lambda\int_{D}a(x)(u_+)^{q-1}w_j dx \\ && - \int_{D}f_n(u_+)w_j dx - \frac{1}{n}\int_{D}\varphi w_j dx, 
\end{array}	\]
	$j=1,2,\dots,m$, and $u=\sum_{i=1}^m\xi_i w_i\in W_m$. Therefore,
	\begin{equation}\label{eq4}
	\begin{array}{lll}
	\left\langle F(\xi),\xi \right\rangle&=&\displaystyle\int_{D}|\nabla u|^N dx+\int_{D}|u|^N dx- \lambda\int_{D}a(x)(u_+)^{q} dx \\ && - \int_{D}f_n(u_+)u_+ dx- \frac{1}{n}\int_{D}\varphi udx,	
	\end{array}
	\end{equation}
	where $u_+ = \max\{u,0\}$, $u_- = u_+ - u$.
	
	Given $u \in W_m$ we define
$$
D^+_n=\{x \in D : |u(x)|\geq \frac{1}{n}\}
$$
and
$$
D^-_n=\{x \in D: |u(x)|< \frac{1}{n}\}.
$$
Thus, we can write \eqref{eq4} as
\[\left\langle F(\xi),\xi\right\rangle = \left\langle F(\xi),\xi\right\rangle_P + \left\langle F(\xi),\xi\right\rangle_{N},\]
where
\[\begin{array}{lll}
\left\langle F(\xi),\xi\right\rangle_P&=&\displaystyle\int_{D^+_n}|\nabla u|^N dx+\int_{D^+_n}|u|^N dx- \lambda\int_{D^+_n}a(x)(u_+)^{q} dx \\ &&-  \int_{D^+_n}f_n(u_+)u_+ dx- \frac{1}{n}\int_{D^+_n}\varphi u dx	
\end{array}\]
and
\[\begin{array}{lll} \left\langle F(\xi),\xi\right\rangle_{N}&=&\displaystyle\int_{D^-_n}|\nabla u|^N dx+\int_{D^-_n}|u|^N dx- \lambda\int_{D^-_n}a(x)(u_+)^{q} dx \\&&- \int_{D^-_n}f_n(u_+)u_+ dx- \frac{1}{n}\int_{D^-_n}\varphi udx.
\end{array}\]

	\textit{\underline{Step 1}}. In what follows, $C$ denotes a generic real constant. Since the embedding $W^{1,N}(\mathbb{R}^N)\subset L^{\tau}(\mathbb{R}^N)$ is continuous for all $\tau\geq N$	(see \cite[Corollary 9.11]{Brezis}), we have
	
	\begin{equation}\label{eq4.1}
	\int_{D^+_n}|a(x)|(u_+)^{q}dx\leq  C\|a\|_{L^{N/(N-q)}(\mathbb{R}^N)}\|\tilde{u}\|^{q}_{W^{1,N}(\mathbb{R}^N)}\leq  \frac{K_1}{2}\|\tilde{u}\|^{q}_{W^{1,N}(\mathbb{R}^N)}.
	\end{equation}
By virtue of Lemmas \ref{lemf} and \ref{lemma2}, we get
	\begin{equation}\label{eq4.2}
	\begin{array}{lll}
	\displaystyle\int_{D^+_n}f_n(u_+)u_+ dx &\leq&
	\displaystyle C_1\int_{D^+_n}|u_+|^p\phi_N(2^{\frac{N}{N-1}}\alpha |u_+|^{\frac{N}{N-1}}) dx\\
	&\leq& \left(\displaystyle\int_{\mathbb{R}^N}|\tilde{u}|^{Np} dx\right)^{\frac{1}{N}} \left(\int_{\mathbb{R}^N}(\phi_N(2^{\frac{N}{N-1}}\alpha |\tilde{u}|^{\frac{N}{N-1}}))^{N'} dx\right)^{\frac{1}{N'}}\\
	&\leq& K_2 \|\tilde{u}\|_{W^{1,N}(\mathbb{R}^N)}^{p} \left(\int_{\mathbb{R}^N}(\phi_N(N2^{\frac{2N}{N-1}}\alpha |\tilde{u}|^{\frac{N}{N-1}})) dx\right)^{\frac{1}{N'}}\\
	&\leq& K_2C(\alpha,N)\|\tilde{u}\|_{W^{1,N}(\mathbb{R}^N)}^{p},
	\end{array}
	\end{equation}
	for $\|\tilde{u}\|_{W^{1,N}(\mathbb{R}^N)}$ small enough, where $\tilde{u}$ is given in \eqref{utilde}.
	 Indeed, if \begin{equation}\label{tm}
	     \|\tilde{u}\|_{W^{1,N}(\mathbb{R}^N)}\leq \frac{1}{4}\left(\frac{\alpha_N}{N \alpha}\right)^{(N-1)/N},\end{equation} then \eqref{TM-N} guarantees the following estimate

	\begin{equation}\label{tm2}\begin{array}{rll}
\left(\int_{\mathbb{R}^N}(\phi_N(N2^{\frac{2N}{N-1}}\alpha |\tilde{u}|^{\frac{N}{N-1}})) dx\right)^{\frac{1}{N'}}&=&\\
	 \left(\int_{\mathbb{R}^N}(\phi_N(N2^{\frac{2N}{N-1}}\alpha \|\tilde{u}\|^{{\frac{N}{N-1}}}_{W^{1,N}(\mathbb{R}^N)}|\frac{\tilde{u}}{\|\tilde{u}\|_{W^{1,N}(\mathbb{R}^N)}}|^{\frac{N}{N-1}})) dx\right)^{\frac{1}{N'}}&\leq& C(\alpha,N).
		\end{array}
	\end{equation}
	
Now, since $\varphi\in L^{N'}(\mathbb{R}^N)$, we have

		\begin{equation}\label{eq4.22}
	\begin{array}{rcl}
\displaystyle	\int_{D^+_n}\varphi udx &\leq&\displaystyle \|\varphi\|_{L^{N'}(\mathbb{R}^N)}\|\tilde{u}\|_{L^{N}(\mathbb{R}^N)}\\
	&\leq&\displaystyle \frac{K_3}{2}\|\tilde{u}\|_{W^{1,N}(\mathbb{R}^N)}.
	\end{array}
		\end{equation}
	
	It follows from \eqref{eq4.1}, \eqref{eq4.2} and \eqref{eq4.22} that
	\begin{equation}\label{eq5}
	\begin{array}{rcl}
	\left\langle F(\xi),\xi\right\rangle_P &\geq &\displaystyle\int_{D^+_n}|\nabla u|^N dx+\int_{D^+_n}|u|^N dx- \lambda \frac{K_1}{2}\|\tilde{u}\|^{q}_{W^{1,N}(\mathbb{R}^N)}\\
	&&- \displaystyle K_2C(\alpha, N)\|\tilde{u}\|^{p}_{W^{1,N}(\mathbb{R}^N)} - \frac{K_3}{2n}\|\tilde{u}\|_{W^{1,N}(\mathbb{R}^N)}.
	\end{array}
	\end{equation}
\begin{remark}\label{est90}
    Notice that the constants $K_1$, $K_2$ and $K_3$ do depend on neither  $|D|=\int_Ddx$ nor $n$.
\end{remark}	
	
	\textit{\underline{Step 2}}. In a similarly way, we obtain 
\begin{equation}\label{eq10}
	\int_{D^-_n}|a(x)|(u_+)^{q}dx\leq   \frac{K_1}{2}\|\tilde{u}\|^{q}_{W^{1,N}(\mathbb{R}^N)}.
	\end{equation}
By virtue of Lemma \ref{lemma2} (ii) we obtain
\begin{equation}\label{eq11}
\begin{array}{lll}
\int_{D^-_n}f_n(u_+)u_+ dx&\leq& C_2 \frac{1}{n^{p-2}}\int_{D^-_n}|u_+|^2 dx\\
&\leq& C_2 \frac{1}{n^{p-2}}\left(\int_{D^-_n} dx \right)^{\frac{N-2}{N}}\left(\int_{\mathbb{R}^N}|\tilde{u}|^N dx \right)^{\frac{2}{N}}\\&\leq& C_2|D|^{(N-2)/N} \frac{1}{n^{p-2}} \|\tilde{u}\|_{W^{1,N}(\mathbb{R}^N)}^2.
\end{array}
\end{equation}
We also have

	\begin{equation}\label{eq4.23}
	\begin{array}{rcl}
\displaystyle	\int_{D^-_n}\varphi udx &\leq&\displaystyle \frac{K_3}{2}\|\tilde{u}\|_{W^{1,N}(\mathbb{R}^N)}.
	\end{array}
		\end{equation}

It follows from \eqref{eq10}, \eqref{eq11} and \eqref{eq4.23} that
\begin{equation}\label{eq12}\begin{array}{lll}
	\left\langle F(\xi),\xi\right\rangle_N &\geq& \displaystyle\int_{D^-_n}|\nabla u|^Ndx+\int_{D^-_n}|u|^N dx- \lambda \frac{K_1}{2}\|\tilde{u}\|_{W^{1,N}(\mathbb{R}^N)}^{q} \\ &&-  \frac{C_2 |D|^{(N-2)/N}}{n^{p-2}}\|\tilde{u}\|^2_{W^{1,N}(\mathbb{R}^N)}-  \frac{K_3}{2n}\|\tilde{u}\|_{W^{1,N}(\mathbb{R}^N)}.
\end{array}
\end{equation}
Using that $\|u\|^N_{W^{1,N}(D)}=\|\tilde{u}\|^N_{W^{1,N}(\mathbb{R}^N)}=\|\nabla \tilde{u}\|^N_{L^N(\mathbb{R}^N)} + \|\tilde{u}\|^N_{L^N(\mathbb{R}^N)}$,  inequalities \eqref{eq5} and \eqref{eq12} imply
\begin{equation}\label{eq13}
\begin{array}{rcl}
\displaystyle\left\langle F(\xi),\xi\right\rangle  &\geq& \displaystyle\|\tilde{u}\|^N_{W^{1,N}(\mathbb{R}^N)} - \lambda K_1\|\tilde{u}\|^{q}_{W^{1,N}(\mathbb{R}^N)}- \displaystyle K_2C(\alpha,N)\|\tilde{u}\|^{p}_{W^{1,N}(\mathbb{R}^N)}\\
&&  -  \frac{C_2|D|^{(N-2)/N}}{n^{p-2}}\|\tilde{u}\|^2_{W^{1,N}(\mathbb{R}^N)} -  \frac{K_3}{n}\|\tilde{u}\|_{W^{1,N}(\mathbb{R}^N)}.
\end{array}
\end{equation}

	Now, let  $|\xi|_m=\|\tilde{u}\|_{W^{1,N}(\mathbb{R}^N)}=\varrho$ for some $\varrho>0$ to be chosen later. Thus, we have
\[	\begin{array}{rcl}
\displaystyle\left\langle F(\xi),\xi\right\rangle  &\geq& \displaystyle \varrho^N - \lambda K_1\varrho^{q}- \displaystyle K_2C(\alpha,N)\varrho^{p}  -  \frac{C_2|D|^{(N-2)/N}}{n^{p-2}}\varrho^2 -  \frac{K_3}{n}\varrho.
\end{array}\]

If $\varrho$ is such that
	\[
	\varrho\leq \frac{1}{(2 K_2 C(\alpha,N))^{\frac{1}{p-N}}},
	\]
	then \[\varrho^N - K_2C(\alpha,N)\varrho^{p} \geq \frac{\varrho^N}{2}.\]

Thus, by choosing \begin{equation}\label{esthro}
    \varrho:=\min\left\{\frac{1}{(2 K_2 C(\alpha,N))^{\frac{1}{p-N}}}, \frac{1}{4}\left(\frac{\alpha_N}{N \alpha}\right)^{(N-1)/N}\right\},
\end{equation} we obtain
	\[
	\begin{array}{rcl}
\displaystyle\left\langle F(\xi),\xi\right\rangle  &\geq& \displaystyle \frac{\varrho^N}{2} - \lambda K_1\varrho^{q}  -  \frac{C_2|D|^{(N-2)/N}}{n^{p-2}}\varrho^2 -  \frac{K_3}{n}\varrho.
\end{array} 
	\]
	Now,  define  $\varsigma:=\frac{\varrho^N}{2} - \lambda K_1\varrho^{q}$. If we choose $$
	\lambda^*:=\frac{\varrho^{N-q}}{4K_1}>0,
	$$
then $\varsigma>\frac{\varrho^N}{4}$ for all $0<\lambda < \lambda^*$.
		Now, we choose $n^* \in \mathbb{N}$ such that
	\[\frac{C_2|D|^{(N-2)/N}}{n^{p-2}}\varrho^2 +  \frac{K_3}{n}\varrho<\frac{\varsigma}{2},\]
	for every $n\geq n^*$. Notice that $n^*$ depends on the domain $D$. Since $\xi \in \mathbb{R}^m$ is such that $|\xi|_m=\varrho$, then for $\lambda < \lambda^*$ and  $n\geq n^*$ we obtain
	\begin{equation}\label{eq8}
	\left\langle F(\xi),\xi\right\rangle \geq \frac{\varsigma}{2}>0.
	\end{equation}
	
	Since  $f_n$ is a Lipschitz function (for each $n \in \mathbb{N}$), it easy to see that $F: \mathbb{R}^m \rightarrow \mathbb{R}^m$ is a continuous function. Thus, for each $\lambda<\lambda^*$ and $n>n^*$ fixed, Lemma \ref{prop1} ensure the existence of   $y \in \mathbb{R}^m$ with $|y|_m\leq \varrho$ and such that $F(y)=0$. In other words, there exists $u_m \in W_m$ verifying
	\begin{equation}\label{est.unif}
	\|u_m\|_{W^{1,N}(D)}\leq \varrho,
	\end{equation}
	and such that
	\begin{equation}\label{eq15}
	\begin{array}{l}\int_{D}|\nabla u_m|^{N-2}\nabla u_m\nabla w dx+ \int_{D}|u_m|^{N-2}u_mwdx=\\ \lambda\int_{D}a(x)(u_{m+})^{q-1}w dx+ \int_{D}f_n(u_{m+})w dx+ \frac{1}{n}\int_{D}\varphi wdx,  
	\end{array}\end{equation}
	for all  $w \in W_m$.
\begin{remark}\label{rem1}
    It is important to mention that $\varrho$, given in \eqref{esthro}, does not depend on the domain $D$, $m$ nor $n$.
\end{remark}	
	Since $W_m \subset W_0^{1,N}(D)$ $\forall \, m \in \mathbb{N}$ and $\varrho$ does not depend on $m$, then $(u_m)_{m\in\mathbb{N}}$ is a bounded sequence in $W_0^{1,N}(D)$. Therefore, for some subsequence, there exists $u \in W_0^{1,N}(D)$ such that
	\begin{equation}\label{eq16}
	u_m \rightharpoonup u \,\,\, \mbox{weakly in} \,\,\, W_0^{1,N}(D),
	\end{equation}
	\begin{equation}\label{eq16.1}
	u_m \to u \,\,\, \mbox{in} \,\,\, L^s(D)\,\,\,s\geq N,
	\end{equation}
	\begin{equation}\label{eq16.1z}
	u_m \to u, \,\,\,  \mbox{a.e. in}\,\,\,D.
	\end{equation}
	
	Thus, 
\begin{equation}\label{n1}
\|u\|_{W^{1,N}(D)} \leq \liminf_{m \to \infty}\|u_m\|_{W^{1,N}(D)}\leq \varrho.
\end{equation}

	
We claim that 
\begin{equation}\label{conv1}
    u_m\rightarrow u \mbox{ in } W^{1,N}_0(D).
\end{equation}
Indeed, using the fact that  $\mathcal{B}=\{w_1,w_2,\dots,w_n,\dots\}$ is a Schauder basis of $W^{1,N}_0(D)$, for every $u \in W^{1,N}_0(D)$  there exists a unique sequence $(\alpha_n)_{n\geq 1}$ in $\mathbb{R}$  such that $u=\sum _{j=1}^{\infty}\alpha_j w_j $, so that
\begin{equation}\label{166}
\psi_m:=\sum_{j=1}^m \alpha _j w_j \rightarrow u \,\, \mbox{ in } W_0^{1,N}(D)\,\,\mbox{ as }m\rightarrow \infty.\end{equation}
Using  $w=(u_m-\psi_m)\in W_m$ as test function in \eqref{eq15}, we obtain
	\begin{equation}\label{eq17z}
	\begin{array}{l}\int_{D}|\nabla u_m|^{N-2}\nabla u_m\nabla(u_m-\psi_m) dx+ \int_{D}|u_m|^{N-2}u_m(u_m-\psi_m)dx=\\ \lambda\int_{D}a(x)(u_{m+})^{q-1}(u_m-\psi_m)dx+ \int_{D}f_n(u_{m+})(u_m-\psi_m) dx\\+ \frac{1}{n}\int_{D}\varphi (u_m-\psi_m)dx.  
	\end{array}
	\end{equation}

From \eqref{eq16}, \eqref{eq16.1} and \eqref{166}, it is easy to see that
\begin{equation}\label{eq113}
\begin{array}{l}
 \int_{D}(|u_m|^{N-1}+|\lambda a(x)(u_{m+})^{q-1}|+ \frac{1}{n}|\varphi|) |u_m-\psi_m|dx\leq \\  (\|u_n\|_{L^{N}(D)}^{N-1}+\lambda\|a\|_{L^{N/(N-q)}(D)}\|u_n\|^{q-1}_{L^{N}(D)}+\frac{1}{n}\|\varphi\|_{L^{N'}(D)})\|u_m-\psi_m\|_{L^{N}(D)}.
\end{array}
\end{equation}

By continuity of $f_n$ and \eqref{eq16.1z} we obtain
 \begin{equation*}
f_n(u_{m_+})^{N'} \to f_n(u_+)^{N'} \,\,\, \mbox{ a.e. in}\,\,\, D.
\end{equation*}
By Lemma \ref{lemma1},  \eqref{est.unif}, and  by using H\"{o}lder inequality we obtain
\begin{equation}
    \begin{array}{lll}
        \displaystyle\int_{D}f_n(u_{m+})^{N'}dx &\leq& \displaystyle c_n^{N'}\int_{D}|u_m|^{N'}dx \\ 
&\leq& c_n^{N'}C|D|^{N/(N-N')}\|u_m\|^{N'}_{W^{1,N}(D)}\\ &\leq& c_n^{N'}C|D|^{N/(N-N')}\varrho^{N'}.
        \end{array}
\end{equation}

Hence, \cite[Theorem 13.44]{HS} leads to 
 \begin{equation}\label{eq16.7}
f_n(u_{m_+}) \to f_n(u_+) \,\,\, \mbox{ weakly in } \,\,\, L^{N'}(D).
\end{equation}
Applying \eqref{eq16.1}, \eqref{166} and \eqref{eq16.7}, we conclude that 
\begin{equation}\label{eq114}
\lim _{m\rightarrow \infty} \displaystyle \int_{D}f_n(u_{m+})(u_m-\psi_m) dx=0.
\end{equation}

By \eqref{est.unif} and  \eqref{166}, we obtain
\begin{equation}\label{eq70z}
\lim_{m\rightarrow\infty} \displaystyle\int_{D}|\nabla u_m|^{N-2}\nabla u_m\nabla (u-\psi_m)dx=0. 
\end{equation}
By \eqref{eq17z},
\eqref{eq113}
\eqref{eq114} and
\eqref{eq70z}, we obtain
\begin{equation}\label{eq71}
\lim_{m\rightarrow\infty} \displaystyle\int_{D}|\nabla u_m|^{N-2}\nabla u_m\nabla (u_m-u)dx=0. 
\end{equation}
Now it is sufficient to apply the $(S_+)-$ property of $-\Delta_N$
(see, e.g., \cite[Proposition 3.5.]{MonMonPapa}) to obtain 
\eqref{conv1}.


	
		Now, for every $m\geq k$ we obtain
	\begin{equation}\label{eq17}
	\begin{array}{l}\int_{D}|\nabla u_m|^{N-2}\nabla u_m\nabla w_k dx+ \int_{D}|u_m|^{N-2}u_mw_kdx=\\ \lambda\int_{D}a(x)(u_{m+})^{q-1}w_k dx+ \int_{D}f_n(u_{m+})w_k dx+ \frac{1}{n}\int_{D}\varphi w_kdx,  
	\end{array}
	\end{equation}
	for all $w_k \in W_k$.
	
	It follows from \eqref{conv1} and \eqref{eq16.7} that

	\begin{equation}\label{eq20}\begin{array}{l}
	\int_{D}|\nabla u|^{N-2}\nabla u\nabla w_k dx+ \int_{D}|u|^{N-2}uw_kdx=\\ \lambda\int_{D}a(x)(u_{+})^{q-1}w_k dx+ \int_{D}f_n(u_{+})w_k dx+ \frac{1}{n}\int_{D}\varphi w_kdx,
	\end{array}\end{equation}  for all $w_k \in W_k$.
	Since $[W_k]_{k \in \mathbb{N}}$ is dense in $W_0^{1,N}(D)$ we conclude that
	\begin{equation}\label{eq21}\begin{array}{l}
	 	\int_{D}|\nabla u|^{N-2}\nabla u\nabla w dx+ \int_{D}|u|^{N-2}uwdx=\\ \lambda\int_{D}a(x)(u_+)^{q-1}w dx+ \int_{D}f_n(u_{+})w dx + \frac{1}{n}\int_{D}\varphi wdx, 
	\end{array}
	\end{equation}
	for all $w \in W_0^{1,N}(D)$. Furthermore, $u\geq 0$ in $D$. In fact, since $u_- \in W_0^{1,N}(D)$ then from \eqref{eq21} we obtain
	\[\begin{array}{lll}
- \|u_-\|^N_{W_0^{1,N}(D)} &=&	\int_{D}|\nabla u|^{N-2}\nabla u\nabla u_- dx+ \int_{D}|u|^{N-2}uu_-dx \\&=& \lambda\int_{D}a(x)(u_{+})^{q-1}u_-dx + \int_{D}f_n(u_{+})u_-dx + \frac{1}{n}\int_{D}\varphi u_-dx\\ &\geq& 0.	\end{array}\]
	Then $u_- \equiv 0$ a.e. in $D$, whence $u\ge 0$ a.e. in $D$. Moreover, $u\not\equiv 0$ is valid
	due to  $\frac{\varphi}{n}> 0$ in $D$.  By applying the strong maximum principle in
	\cite[Theorem 5.4.1]{PS}  we obtain  $u>0$ in $D$, and   \cite[Theorem
	5.5.1]{PS} ensure that $\partial u/\partial \nu<0$ on
	$\partial D$ holds. By Lemma \ref{lemma1} and \cite[Theorem 7.1]{LU} we conclude that $u\in L^{\infty}(D)$. Thus, \cite[Theorem 1]{lieberman} and \cite[p. 320]{L} ensure the regularity up to the boundary $u\in C^{1,\beta}(\overline{D})$, for some $\beta \in (0,1)$.

Therefore, we conclude that proof of the lemma by taking $u_n=u$.

\end{proof} 


\subsection{ Proof of Theorem~\ref{TP2}}
First we show that $(PD)$ has a positive solution. For each $n \in \mathbb{N}$, $n> n^*$, by Lemma \ref{lem:appro-sol},  equation $(PD_n)$ has a  solution $u_n \in W^{1,N}_{0}(D) \cap C^{1,\beta}(\overline{D})$. Thus
\begin{equation}\label{eq70}
	\begin{array}{l}
	\displaystyle \int_{D}|\nabla u_n|^{N-2}\nabla u_n\nabla wdx + \int_{D}|u_n|^{N-2}u_n wdx = \\ \displaystyle \lambda\int_{D}a(x)u_n^{q}w dx+ \int_{D}f_n(u_n)wdx + \frac{1}{n}\int_{D}\varphi wdx,\end{array}
	\end{equation}
	for all  $w \in W_0^{1,N}(D)$.

By \eqref{n1} we have that
\begin{equation}\label{c12}
\|u_n\|_{W^{1,N}(D)} \leq \varrho, \,\, \forall\, n \in \mathbb{N},
\end{equation}
and $\varrho$ does not depend on $n$. Thus, along a subsequence again relabeled as $u_n$, there exists $u \in W^{1,N}_{0}(D)$ such that
\begin{equation}\label{eq23}
u_n \rightharpoonup u \,\,\, \mbox{weakly in} \,\,\, W^{1,N}_{0}(D), \,\, \mbox{as} \,\, n \rightarrow \infty.
\end{equation}

	Thus, 
\begin{equation}\label{l11}
\|u\|_{W^{1,N}(D)} \leq \liminf_{n \to \infty}\|u_n\|_{W^{1,N}(D)}\leq \varrho.
\end{equation}

We claim that 
\begin{equation}\label{eq23zl}
u_n \rightarrow u \,\,\,  W^{1,N}_{0}(D), \,\, \mbox{as} \,\, n \rightarrow \infty.
\end{equation}


In fact, the proof of \eqref{eq23zl}  follows in a similarly way  as we did in the previous section.

First, notice that
 Lemma \ref{lemma1} and \eqref{c12} imply 
\begin{equation}
    \begin{array}{lll}
        \displaystyle\int_{D}f_n(u_{n})^{N'}dx &\leq& \displaystyle c_n^{N'}\int_{D}|u_n|^{N'}dx
\leq  c_n^{N'}C\varrho^{N'}.
        \end{array}
\end{equation}
Moreover, 
\[
u_n \rightarrow u\,\, \mbox{a.e. in}\,\, D,
\]
and by the uniform convergence of Lemma \ref{lemma1} ($iii$) we have
\begin{equation}\label{eq26.2zl}
f_n(u_n(\cdot)) \rightarrow f(u(\cdot))\,\, \mbox{a.e. in}\,\, D.
\end{equation}
Hence, \cite[Theorem 13.44]{HS} leads to 
 \begin{equation}\label{eq16.7lz}
f_n(u_{n}) \to f(u) \,\,\, \mbox{ weakly in } \,\,\, L^{N'}(D).
\end{equation}
On the other hand, taking $w=(u_n-u)$ as a test function in \eqref{eq70}, we get

\begin{equation}\label{est7z}
		\begin{array}{ll}
	&\displaystyle \int_{D}|\nabla u_n|^{N-2}\nabla u_n\nabla (u_n-u)dx \\  
	\leq & - \int_{D}|u_n|^{N-2}u_n (u_n-u)dx  + \lambda\int_{D}a(x)u_n^{q-1}(u_n-u) dx \displaystyle\\
	& + \int_{D}f_n(u_n)(u_n-u)dx+ \frac{1}{n}\int_{D}\varphi (u_n-u)dx\\
	\leq& \|u_n\|^{N-1}_{L^{N}(D)}\|u_n-u\|_{L^{N}(D)} \\ 
	& +\lambda\|a\|_{L^{N/(N-q)}(D)}\|u_n\|^{q-1}_{L^{N}(D)}\|u_n-u\|_{L^{N}(D)}\\  & +\|f_n(u_n)\|_{L^{N'}(D)}\|u_n-u\|_{L^{N}(D)}\\ & +\frac{1}{n}\|\varphi\|_{L^{N'}(D)}\|u_n-u\|_{L^{N}(D)} 
		\rightarrow 0 \,\, \mbox{ as }n\rightarrow\infty.
	\end{array}
\end{equation}
And then, $\limsup_{n\rightarrow \infty}\int_{D}|\nabla u_n|^{N-2}\nabla u_n\nabla (u_n-u)dx \leq 0$. Therefore, \eqref{eq23zl} follows by  $(S_+)$ property.

	By using \eqref{eq23zl} and \eqref{eq16.7lz},  we can pass to the limit in \eqref{eq70} to obtain

\begin{equation}\label{est7}
		\begin{array}{l}
	\displaystyle \int_{D}|\nabla u|^{N-2}\nabla u\nabla wdx + \int_{D}|u|^{N-2}u wdx =  \displaystyle \lambda\int_{D}a(x)u^{q-1}w dx+ \int_{D}f(u)wdx,\end{array}
\end{equation}
	for all  $w \in W_0^{1,N}(D)$.
Thus,  $u$ is a  solution of $(PD)$.


Now,  fix a positive constant $\lambda$ such that
\begin{equation}\label{eq:b}
\lambda< \lambda^*=\frac{\varrho^{N-q}}{4K_1}.
\end{equation}
Since $a>0$ is a continuous function, define 
$$a_D=\inf_{D} a(x).$$
Then, according to Lemma \ref{le2}, there exists a
positive solution $u_{0}$ of
$$
\left\{
\begin{array}{lll}
-\Delta_Nu +|u|^{N-2}u=\lambda a_Du^{q-1} &\mbox{in}&D,\\
u>0&\mbox{in}&  D,\\
u=0 & \mbox{on} & \partial D.
\end{array} \right.
$$

Let $u_n$ be a positive solution of problem $(PD_n)$
obtained by Lemma~\ref{lem:appro-sol}. We observe that
$u_{0}/u_n, u_n/u_{0} \in L^\infty(D)$ because $u_{0}$
and $u_n$ are positive functions belonging to
$C^{1,\beta}_0(\overline{D})$ and satisfying $\partial u_n/\partial \nu<0$, $\partial u_{0}/\partial \nu<0$
on $\partial D$.
 Notice that
\begin{equation*}
\lambda a(x)t^{q-1} + f_n(t)+\frac{\varphi}{n}
\geq \lambda a_D t^{q-1}=g(t).
\end{equation*}
Hence, $u_1=u_{0}$ and $u_2=u_n$ are a positive subsolution and a positive supersolution of
problem \eqref{19}, respectively. Thus, by Proposition \ref{com:Dirichlet} we see that
$u_n\ge u_{0}$ in $D$ for every $n>n^*$.
Therefore, by passing to the limit we obtain  $$u\geq u_{0} \mbox{ a.e. in }D.$$ Thus, we conclude that $u$ is a positive
solution of problem $(PD)$.


From now, the solution we just found will be labeled as $u_\lambda$ with explicit dependence on $\lambda$. In what follows, we will deduce that $\|u_{\lambda}\|_{W^{1,N}(D)} \to 0$ as $\lambda \to 0$. Fix the pair $(\lambda,u_{\lambda})$, where  $\lambda$ $\in(0,\lambda^*)$ and  $u_{\lambda}$ is the corresponding solution of problem $(PD)$. By using $w=u_{\lambda}$ as a test function in \eqref{est7}, we obtain

\begin{equation}\label{est8}
		\begin{array}{lll}
	\displaystyle \int_{\mathbb{R}^N}|\nabla \tilde{u}_{\lambda}|^Ndx + \int_{\mathbb{R}^N}\tilde{u}_{\lambda}^N dx &=& 
	\displaystyle \int_{D}|\nabla u_{\lambda}|^Ndx + \int_{D}u_{\lambda}^N dx\\
	 &=&  \displaystyle \lambda\int_{D}a(x)u_{\lambda}^{q} dx+ \int_{D}f(u_{\lambda})u_{\lambda}dx\\
	 &=&  \displaystyle \lambda\int_{\mathbb{R}^N}a(x)\tilde{u}_{\lambda}^{q} dx+ \int_{\mathbb{R}^N}f(\tilde{u}_{\lambda})\tilde{u}_{\lambda}dx\\
	&\leq& \lambda K_1 \|\tilde{u}_{\lambda}\|^{q}_{W^{1,N}(\mathbb{R}^N)} +K_2 C(\alpha, N) \|\tilde{u}_{\lambda}\|^{p}_{W^{1,N}(\mathbb{R}^N)} ,
	 \end{array}
\end{equation}
where $K_1$, $K_2$ are given in \eqref{eq4.1}, 
\eqref{eq4.2}, respectively.
	
	 Since $\tilde{u}_{\lambda} \neq 0$, from \eqref{est8}, we have the following estimate

\begin{equation}\label{est9}
		\begin{array}{lll}
\|\tilde{u}_{\lambda}\|^{N-q}_{W^{1,N}(\mathbb{R}^N)} (1-	K_2C(\alpha,N) \|\tilde{u}_{\lambda}\|^{p -N}_{W^{1,N}(\mathbb{R}^N)}) &\leq& \lambda K_1.
	 \end{array}
\end{equation}
By combining  \eqref{esthro} and \eqref{c12}, we obtain  
$$  \|\tilde{u}_{\lambda}\|^{p -N}_{W^{1,N}(\mathbb{R}^N)} \leq  \frac{1}{2K_2 C(\alpha, N)}.$$
Thus,
\begin{equation}\label{est10}
		\begin{array}{lll}
\|u_{\lambda}\|_{W^{1,N}(D)} =\| \tilde{u}_{\lambda}\|_{W^{1,N}(\mathbb{R}^N)}  &\leq& (2 \lambda K_1)^{1/(N-q)} \to 0 \quad as \quad \lambda \to 0.
	 \end{array}
\end{equation}
Thus, the proof of the theorem is complete.

 


\section{Proof of the Main Theorem}\label{S10}
\subsection{A priori estimates}\label{S5-1}

In this subsection, for convenience, when necessary, we will omit the notation $\tilde{u}$. In order to prove Theorem  \ref{TP} and \ref{T:infinity}, it will be needed a couple of estimates proven in the following results. It is important to mention that much less is known about the results of regularity for the $L_p$ operator, as can be seen in  \cite{Isaia}. This becomes an obstacle to obtain uniform estimates in the sense of H\"{o}lder norm.

 Fix  $u\in W^{1,N}(D)$ any positive solution of $(PD)$ given by Theorem \ref{TP2}. Here, we will borrow some ideas from \cite{FMT}. Define $u_M:=\min\{u,M\}$ for $M>0$.
Choose $\overline{p}^*$ satisfying $2N^2<\overline{p}^*$. For $R^\prime>R>0$, we take a smooth function $\eta_{R,R^\prime}$
such that $0\le \eta_{R,R^\prime}\le 1$,
$\|\eta_{R,R^\prime}^\prime\|_\infty\le 2/(R^\prime-R)$,
$\eta_{R,R^\prime}(t)=1$ if $t\le R$ and
$\eta_{R,R^\prime}=0$ if $t\ge R^\prime$.

\begin{lemma}\label{lem:element}
	Let $x_0\in\mathbb{R}^N$, $M>0$, $R^\prime>R>0$, such that $B(x_0,R')\subset D$, 
	$\gamma_1=\frac{N}{N-q}>1$ and $\gamma_1^\prime=\frac{N}{q}$ such that $1/\gamma_1+1/\gamma_1^\prime=1$.
	Denote $\eta(x):=\eta_{R,R^\prime}(|x-x_0|)$.
	Assume that $2N\le \tilde{p}$ (in particular $\gamma_1^\prime\le \tilde{p}$) and $u  \in W^{1,N}(D)$ a solution of ($PD$), in particular, $u\in L^{\tilde{p}(N+\beta)}(B(x_0,R^\prime))$ with $\beta\ge 0$.
	Then it holds:
	\begin{eqnarray}
	\lefteqn{\int_{B(x_0,R^\prime)} f(u)uu_M^{\beta}\eta^{N}dx\,  
		\le C(\varrho)
		\|u\|^{\beta}_{L^{\tilde{p}(N+\beta)}(B(x_0,R^\prime))}B_{R^\prime}}
	\label{eq:element-1}\\
	&&\int_{B(x_0,R^\prime)} a(x) u^{q}u_M^{\beta}\eta^{N}dx\,  
	\le \|a\|_{L^{\gamma_1}(B(x_0,R^\prime))}
	\|u\|_{L^{\tilde{p}(N+\beta)}(B(x_0,R^\prime))}^{q+\beta}
	B_{R^\prime}
	\label{eq:element-2}
	\end{eqnarray}
	where $B_{R^\prime} :=(1+|B(0,R^\prime)|)$ and
	$|B(0,R^\prime)|$ denotes the Lebesgue measure of the ball $B(0,R^\prime)$.
	
\end{lemma}
\begin{proof} 
	According to H\"older's inequality, we easily show \eqref{eq:element-2}.
	So, we will prove only \eqref{eq:element-1}.
	By Young's inequality, Lemmas \ref{lemf}, \ref{lemma2}, and inequality \eqref{tm}, we obtain
		\[
		\begin{array}{l}
		\int_{B(x_0,R^\prime)} f(u)uu_M^{\beta}\eta^{N}dx\, \\ 
		\le C
		\int_{B(x_0,R^\prime)} |u|^{p}u_M^{\beta}\phi_N(\alpha|u|^{\frac{N}{N-1}})dx\,  
		\\
		\le C
		\left(\int_{B(x_0,R^\prime)} |u|^{Np}u_M^{N\beta}dx\right)^{1/N}\left(\int_{B(x_0,R^\prime)}[\phi_N(\alpha|u|^{\frac{N}{N-1}})]^{N'}dx\right)^{1/N'}\,\\
		\le C
		\left(\int_{B(x_0,R^\prime)} |u|^{2Np}dx\right)^{1/2N}\left(\int_{B(x_0,R^\prime)}u_M^{2N\beta}dx\right)^{1/2N}C(\alpha,N)\,\\
		\le C
		\|u\|^p_{L^{2Np}(B(x_0,R^\prime))}\|u\|^{\beta}_{L^{2N\beta}(B(x_0,R^\prime))}\,\\
		\le C(\varrho)
		\|u\|^{\beta}_{L^{2N\beta}(B(x_0,R^\prime))}\,,
	\end{array}
	\]
because $\|u\|_{L^{4p}(B(x_0,R^\prime))} \le C\|\tilde{u}\|_{W^{1,N}(\mathbb{R}^N)}\leq C \varrho$. 

Since $\tilde{p}(N+\beta)>2N\beta$, by H\"older's inequality we obtain
\[
\left(\int_{B(x_0,R^\prime)}u^{2N\beta}dx\right)^{1/2N}\le \|u\|^{\beta}_{L^{\tilde{p}(N+\beta)}(B(x_0,R^\prime))}|B(x_0,R^\prime)|^{\frac{\tilde{p}(N+\beta)-2N\beta}{2N\tilde{p}(N+\beta)}},
\]
that conclude the inequality \eqref{eq:element-1}.
\end{proof}

\begin{lemma}\label{lem:bdd}
	Let $x_0\in\mathbb{R}^N$, $R^\prime>R>0$, such that $B(x_0,R')\subset D$, $\gamma_1=\frac{N}{N-q}>1$ and $\gamma_1^\prime=\frac{N}{q}$.
	Assume that $2N\le \tilde{p}$.
	As $u\in L^{\tilde{p}(N+\beta)}(B(x_0,R^\prime))$ with $\beta\ge 0$, then
	\begin{eqnarray}\label{eq:iteration}
	\lefteqn{\|u\|_{L^{\frac{\overline{p}^*}{N}(N+\beta)}(B(x_0,R))}^{N+\beta} }\\ &&
	\le 2^N(N+\beta)^NC_*^NB_{R^\prime}(C_{R^\prime}+D_{R,R^\prime})
	\max\{1,\|u\|_{L^{\tilde{p}(N+\beta)}(B(x_0,R^\prime))}\}^{N+\beta} \nonumber
	\end{eqnarray}
	holds with
	$$\begin{array}{l}
	B_{R^\prime} :=1+|B(0,R^\prime)|, \\
	C_{R^\prime} :=
	C(\varrho)+\lambda^*\|a\|_{L^{\gamma_1}(B(x_0,R^\prime))},
	\\
	D_{R,R^\prime} := \frac{N^N2^{2N-1}+2^{N-1}}{(R^\prime-R)^N},
	\end{array}
	$$
	where $C_*$ is the positive constant embedding from
	$W^{1,N}(\mathbb{R}^N)$ to $L^{\overline{p}^*}(\mathbb{R}^N)$.
\end{lemma}
\begin{proof} 
	Taking $uu_M^{\beta}\eta^{N}\in W^{1,N}_0(B(x_0,R^\prime))$ (for $M>0$)
	as test function in \eqref{est7}, where
	$\eta(x)=\eta_{R,R^\prime}(|x-x_0|)$, and
	by Lemma~\ref{lem:element},
	we obtain
	\begin{eqnarray}
	\lefteqn{C(\varrho)
		\|u\|^{\beta}_{L^{\tilde{p}(N+\beta)}(B(x_0,R^\prime))}B_{R^\prime} }
	\nonumber \\ &&
	+\lambda\|a\|_{L^{\gamma_1}(B(x_0,R^\prime))}
	\|u\|_{L^{\tilde{p}(N+\beta)}(B(x_0,R^\prime))}^{q+\beta}
	B_{R^\prime}
	\nonumber \\
		&\ge&
	\int_{B(x_0,R^\prime)} |\nabla u|^Nu_M^\beta\eta^{N}dx\,  
	+\int_{B(x_0,R^\prime)}u_M^{N+\beta}\eta^{N}dx\, \label{eq:bdd-1.1} 
	\\
	&&-\frac{2N}{R^\prime-R}
	\int_{B(x_0,R^\prime)} |\nabla u|^{N-1}
	u_M^\beta u\eta^{N-1}dx,\nonumber
	\end{eqnarray}
	where we use $|\nabla \eta|\le 2/(R^\prime-R)$.
	From Young's inequality and H\"older's inequality, we obtain
	\begin{eqnarray}
	\lefteqn{\frac{2N}{R^\prime-R}
		\int_{B(x_0,R^\prime)} |\nabla u|^{N-1}
		u_M^\beta u\eta^{N-1}dx\,  }
	\nonumber \\
	&\le& \frac{1}{2}\int_{B(x_0,R^\prime)} |\nabla u|^N
	u_M^\beta\eta^{N}dx\,  
	+\frac{2^NN^N2^{N-1}}{(R^\prime-R)^N}
	\int_{B(x_0,R^\prime)}u^{N+\beta}dx\,  
	\nonumber \\
	&\le& \frac{1}{2}\int_{B(x_0,R^\prime)} |\nabla u|^N
	u_M^\beta\eta^{N}dx\,   +\frac{N^N2^{2N-1}}{(R^\prime-R)^N}
	\|u\|_{L^{\tilde{p}(N+\beta)}(B(x_0,R^\prime))}^{N+\beta}
	B_{R^\prime}.
	\label{eq:bdd-2}
	\end{eqnarray}
	Thus, from \eqref{eq:bdd-1.1} and \eqref{eq:bdd-2} we have
	\begin{eqnarray}
	\lefteqn{B_{R^\prime}\left(C_{R^\prime}
		+\frac{N^N2^{2N-1}}{(R^\prime-R)^N}\right)
		\max\{1,\|u\|_{L^{\tilde{p}(N+\beta)}(B(x_0,R^\prime))}\}^{N+\beta} }
	\nonumber\\
	&\ge&
	\frac{1}{2}\int_{B(x_0,R^\prime)} |\nabla u|^Nu_M^\beta\eta^{N}dx\,  
	+\int_{B(x_0,R^\prime)}u_M^{N+\beta}\eta^{N}dx\,  .
	\label{eq:bdd-3}
	\end{eqnarray}
	Moreover, by using
	\begin{eqnarray*}
		\lefteqn{\|\nabla (u_M^{1+\beta/N}\eta)\|_{L^N(\mathbb{R}^N)}^N
			\le 2^{N-1}\left\{\|\eta \nabla (u_M^{1+\beta/N})\|_{L^N(\mathbb{R}^N)}^N
			+\|u_M^{1+\beta/N}\nabla \eta\|_{L^N(\mathbb{R}^N)}^N \right\} }
		\\
		&\le& 2^{N-1}\left(1+\frac{\beta}{N}\right)^N
		\int_{B(x_0,R^\prime)} |\nabla u|^Nu_M^\beta\eta^{N}dx\,  
		+\frac{2^{2N-1}}{(R^\prime-R)^N}
		\int_{B(x_0,R^\prime)}u_M^{N+\beta}dx\,  
	\end{eqnarray*}
	and H\"older's inequality, due to the embedding from
	$W^{1,N}(\mathbb{R}^N)$ to $L^{\overline{p}^*}(\mathbb{R}^N)$, we have 
	\begin{equation}
	\label{eq:bdd-4}
	\begin{array}{l}
	\displaystyle \frac{1}{2}\int_{B(x_0,R^\prime)} |\nabla u|^Nu_M^\beta\eta^{N}dx\,  
	+\int_{B(x_0,R^\prime)}u_M^{N+\beta}\eta^{N}dx\,        \\  \\
	\ge 2^{-N}N^N(N+\beta)^{-N}
	\left\{\|\nabla (u_M^{1+\beta/N}\eta)\|_{L^N(\mathbb{R}^N)}^N
	+\|u_M^{1+\beta/N}\eta\|_{L^N(\mathbb{R}^N)}^N \right\}
	 \\  \\
	\qquad -\frac{2^{N-1}N^N}{(N+\beta)^{N}(R^\prime-R)^N}
	\int_{B(x_0,R^\prime)}u_M^{N+\beta}dx\,  
	 \\ \\
	\ge 2^{-N}N^N(N+\beta)^{-N}
	\|u_M^{1+\beta/N}\eta\|_{W^{1,N}(\mathbb{R}^N)}^N \\  \\
	\qquad -\frac{2^{N-1}}{(R^\prime-R)^N}
	\|u\|_{L^{\tilde{p}(N+\beta)}(B(x_0,R^\prime))}^{N+\beta}(1+|B(0,R^\prime)|)
	 \\ \\
	\ge  2^{-N}N^N(N+\beta)^{-N}C_*^{-N}
	\|u_M^{1+\beta/N}\eta\|_{L^{\overline{p}^*}(\mathbb{R}^N)}^N
\\ \qquad	-\frac{2^{N-1}}{(R^\prime-R)^N}
	\|u\|_{L^{\tilde{p}(N+\beta)}(B(x_0,R^\prime))}^{N+\beta}B_{R^\prime}
	\\ \\
	\ge 2^{-N}N^N(N+\beta)^{-N}C_*^{-N}
	\|u_M\|_{L^{\overline{p}^*(N+\beta)/N}(B(x_0,R))}^{N+\beta}
	\\  \\
	\qquad -\frac{2^{N-1}}{(R^\prime-R)^N}
	\|u\|_{L^{\tilde{p}(N+\beta)}(B(x_0,R^\prime))}^{N+\beta} B_{R^\prime}.
		\end{array}
		\end{equation}
	Consequently, it follows from \eqref{eq:bdd-3} and \eqref{eq:bdd-4}  that
	\begin{eqnarray}
	\lefteqn{2^{-N}N^N(N+\beta)^{-N}C_*^{-N}
		\|u_M\|_{L^{\overline{p}^*(N+\beta)/N}(B(x_0,R))}^{N+\beta}}
	\nonumber \\
	&\le& B_{R^\prime}(C_{R^\prime}+D_{R,R^\prime})
	\max\{1,\|u\|_{L^{\tilde{p}(N+\beta)}(B(x_0,R^\prime))}\}^{N+\beta}.
	\label{eq:bdd-5}
	\end{eqnarray}
	The conclusion follows by applying Fatou's lemma and letting $M\to\infty$ in \eqref{eq:bdd-5}.
\end{proof}

\begin{proposition}\label{prop:bdd-2}
Assume the assumptions of Lemma \ref{lem:bdd}. Let us suppose $x_0 \in D$ and $a$ a positive function such that $a\in L^{\frac{N}{N-q}}(\mathbb{R}^N)$, $f$ satisfying \eqref{growth}, $R_*>0$ such that $B(x_0,2R_*) \subset D$. If $u  \in W^{1,N}(D)$ is a solution of ($PD$), then 
	$u\in L^\infty(D)$.  Furthermore, $\|u\|_{W^{1,N}(D)}\leq C$ implies $\|u\|_{L^{\infty}(D)}\leq \tilde{C}$.
\end{proposition}
\begin{proof}
	
	Since $2N<\overline{p}^*/N$,
	we  can choose $\tilde{p}$ such that
	$$
		2N<\tilde{p}<\frac{\overline{p}^*}{N}.
	$$

	Let $R_*>0$ satisfying $B(x_0,2R_*) \subset D$.
	Put
	$$
	A:=\lambda^*\|a\|_{L^{\gamma_1}(\mathbb{R}^N)}.
	$$
	Define sequences $\{\beta_n\}$, $\{R_n^\prime\}$ and $\{R_n\}$ by
	\begin{eqnarray*}
		\lefteqn{\beta_0:=\frac{\overline{p}^*}{\tilde{p}}-N>0,
			\quad \tilde{p}(N+\beta_{n+1})=\frac{\overline{p}^*}{N}
			(N+\beta_n),}
		\\
		&& R_n^\prime:=(1+2^{-n})R_*, \quad R_n:=R_{n+1}^\prime.
	\end{eqnarray*}
	Since $u\in W^{1,N}(D)$,
	by using the embedding of $W^{1,N}(D)$ to
	$L^{\overline{p}^*}(D)$,
	we see that
	$u\in L^{\overline{p}^*}(D)=L^{\tilde{p}(N+\beta_0)}(D)$.
	
	Let $x_0\in D$ fixed.
	 Lemma~\ref{lem:bdd} guarantees that
	if $u\in L^{\tilde{p}(N+\beta_n)}(B(x_0,R_n^\prime))$, then
	$u\in L^{\frac{\overline{p}^*}{N}(N+\beta_n)}(B(x_0,R_n))
	=L^{\tilde{p}(N+\beta_{n+1})}(B(x_0,R_{n+1}^\prime))$.
	Notice that
	$$
	\begin{array}{ll}
	B_{R^\prime_n} &\le (1+|B(0,2R_*)|)=:B_0, \\
	C_{R_n^\prime} &\le C(\varrho)+A+1=:C_0, \\
	D_{R_n,R_n^\prime} &= \frac{N^N2^{2N-1}+2^{N-1}}{R_*^N}2^{N(n+1)}=C^\prime 2^{N(n+1)}=:D_n
	\end{array}
	$$
	for any $n\ge 0$ with $C^\prime$ independent of $n$.
	By setting $$b_n:=\max\{1,\|u\|_{L^{\tilde{p}(N+\beta_n)(B(x_0,R_n^\prime))}}\},$$
and	by Lemma~\ref{lem:bdd} we obtain
	\begin{equation}\label{eq:bdd-2-1}
	b_{n+1}\le C^{\frac{1}{N+\beta_n}}(N+\beta_n)^{\frac{N}{N+\beta_n}}
	(C_0+D_n)^{\frac{1}{N+\beta_n}}b_n
	\end{equation}
	for every $n\ge0$
	with $\ C:=2^N(C_*+1)^NB_0$. Put $P:=\tilde{p}N/\overline{p}^*<1$.
	Then,  $N+\beta_{n+1}=(N+\beta_n)/P$,
	$\beta_{n+1}>\beta_n/P>\beta_0(1/P)^{n+1}\to\infty$ as $n\to\infty$.
	Moreover, we see that
	\begin{align*}
	S_1 &:=\sum_{n=0}^\infty\frac{1}{N+\beta_n}=
	\frac{1}{N+\beta_0}\sum_{n=0}^\infty P^n
	=\frac{1}{(N+\beta_0)(1-P)}<\infty,
	\\
	S_2 &:=\mbox{ln} \prod_{n=0}^\infty (N+\beta_n)^{\frac{N}{N+\beta_n}}
	=\frac{N}{N+\beta_0}\sum_{n=0}^\infty P^n
	\left(\mbox{ln} (N+\beta_0)+n\mbox{ln} P^{-1}\right)<\infty
	\end{align*}
	and
	\begin{align*}
	S_3&:=\mbox{ln} \prod_{n=0}^\infty(C_0+D_n)^{\frac{1}{N+\beta_n}}
	=\sum_{n=0}^\infty \frac{P^{n}}{N+\beta_0} \mbox{ln} (C_0+D_n)
	\\
	&\le \sum_{n=0}^\infty\frac{P^{n}}{N+\beta_0}
	N(n+1)\mbox{ln} (C_0+C^\prime)2<\infty.
	\end{align*}
	As a result, by iteration in \eqref{eq:bdd-2-1}  and
	$\tilde{p}(N+\beta_0)=\overline{p}^*$, we obtain
	$$
	\|u\|_{L^{\frac{\overline{p}^*}{N}(N+\beta_n)}(B(x_0,R_*))}
	\le b_n \le C^{S_1}e^{S_2}e^{S_3}
	\max\{1,\|u\|_{L^{\overline{p}^*}(B(x_0,2R_*))}\}
	$$
	for every $n\ge 1$.
	Letting $n\to\infty$, this ensures that
	\begin{equation}\label{eq:bdd-2-10}
	\|u\|_{L^{\infty}(B(x_0,R_*))}
	\le C^{S_1}e^{S_2}e^{S_3}
	\max\{1,\|u\|_{L^{\overline{p}^*}(B(x_0,2R_*))}\}.
	\end{equation}
	By using the embedding of $W^{1,N}(D)$ to
	$L^{\overline{p}^*}(D)$,
	\eqref{eq:bdd-2-10} yields that
	\begin{equation}\label{z7}
	   \begin{array}{ll}
	\|u\|_{L^\infty(B(x_0,R_*))}
	&\le C^{S_1}e^{S_2}e^{S_3}
	\max\{1,\|u\|_{L^{\overline{p}^*}(D)}\}
	\\
	&\le C^{S_1}e^{S_2}e^{S_3}
	\max\{1,C_*\|u\|_{W^{1,N}(D)}\}
	\end{array}
	\end{equation}
whence $u$ is bounded in $D$ because
	$x_0\in D$ is arbitrary and
	the constant $C^{S_1}e^{pS_2}e^{S_3}$ is independent of $x_0$.
\end{proof}


\subsection{Proof of Theorem \ref{TP}}\label{S4}

In this section, we denote $B_n:=B_n(0)$ the open ball centered at the origin with radius $n$. Throughout this section, we will consider $\lambda \in(0,\lambda ^*)$ fixed. The space $W^{1,N}(B_n)$ is endowed with the norm
$$
\|u\|_{N,n}^N :=\int_{B_n}  \left(|\nabla u|^{N}+|u|^{N}\right)dx
\,  .
$$

By applying Theorem~\ref{TP2} with $D=B_n$ ($n\in\mathbb{N}$), we
obtain a solution $u_n\in W_0^{1,N}(B_n)$ of the
problem

$$\left\{
\begin{array}{lll}
-\Delta_N u+|u|^{N-2}u=\lambda a(x)|u|^{q-2}u+f(u) &\mbox{in}&   B_n\nonumber\\
u>0&\mbox{in}&   B_n\nonumber\\
u(x)= 0 &\mbox{on}&\partial B_n.\nonumber
\end{array} \right.\leqno(P_n)
$$

Again, \eqref{l11} and \eqref{est10} shows
the boundedness of $\{u_n\}_{n>0}$
in $W_0^{1,N}(D)$. That is,
\begin{equation}
\|u_n\|_{N,n}\  \le  \tilde{\varrho}\quad \mbox{for all}\quad n\in\mathbb{N}, 
\label{eq:apriori-p}
\end{equation}
where $$\tilde{\varrho}:=\min \left\{(2 \lambda C_1)^{1/(N-q)}, \varrho \right\}$$  is independent of $B_n$. 

 If $n\geq m+1$, notice that
\begin{equation}\label{z2}
\begin{array}{lll}
    \int_{B_{m+1}}(|\nabla u_n|^{N-2}\nabla u_n \nabla\varphi+|u_n|^{N-2}u_n\varphi-\lambda a(x)|u_n|^{q-2}u_n\varphi-f(u_n)\varphi)dx=0,
\end{array}
\end{equation}
for all $\varphi\in C^{\infty}_0(B_{m+1})$.
By \eqref{eq:apriori-p}  we
obtain
\begin{equation}\label{20}
\|u_n\|_{N,m+1}\leq \|u_n\|_{N,n}\le \tilde{\varrho}.
\end{equation}
\noindent Therefore, there exists $u_{\lambda}\in W^{1,N}(B_{m+1})$ such that
\begin{eqnarray}\label{z1}
u_{n} &\rightharpoonup& u_{\lambda} \quad \,\,\, \mbox{in} \,\, W^{1,N}(B_{m+1}),
\label{eq:main-1}\\
u_n&\to &u_{\lambda} \quad  {\rm in}\ L^s(B_{m+1}),\, s\geq N,
\label{eq:main-2}\\
u_{n}(x)&\rightarrow& u_{\lambda}(x)\quad \,\,\,  \mbox{ a.e. } x\in B_{m+1},
\end{eqnarray}
as $n\to\infty$.

By inequality \eqref{z7} in Proposition \ref{prop:bdd-2} and \eqref{20}, we infer that
\begin{equation}\label{z3}
	\begin{array}{ll}
	\|u_n\|_{L^\infty(B_{m+1})} 
	&\leq  C^{S_1}e^{S_2}e^{S_3}
	\max\{1,C_*\tilde{\varrho}\}=:\Theta.
	\end{array}
    \end{equation}
Now we use regularity result up to the boundary due to Lieberman \cite[Theorem 1.7]{L}, to conclude from \eqref{z3} that 
\begin{equation}\label{z4}
    \|u_n\|_{C^{1,\beta}(\overline{B_m})}\leq\vartheta, 
\end{equation}
where $\beta\in (0,1)$ and $\vartheta$ is independent of $n$.
Thus, using \eqref{z4} and Arzel\`a-Ascoli theorem, we conclude that \begin{equation}\label{j1}
    u_\lambda\in C^{1,\alpha}(\overline{B_m}) \mbox{ for some }\alpha\in(0,\beta).
\end{equation}
We also have
\begin{equation}\label{eqzx1} 
    \begin{array}{l}
f(u_{n}) \to f(u_{\lambda}) \,\,\, \mbox{ weakly in } \,\,\, L^{N'}(B_{m+1}),
   \end{array}
\end{equation}
as $n\to\infty$. Indeed, notice that

	\begin{equation}\begin{array}{rll}
\int_{B_{m+1}} |f(u_n)|^{N'}dx\leq\int_{B_{m+1}}(u_n^{(p-1)}\phi_N(2^{\frac{N}{N-1}}\alpha |u_n|^{\frac{N}{N-1}}))^{N'} dx&\leq&\\
	 \Theta^{(p-1)N'}\int_{B_{m+1}}(\phi_N(N2^{\frac{2N}{N-1}}\alpha \|u_n\|^{{\frac{N}{N-1}}}_{W^{1,N}(B_n}|\frac{u_n}{\|u_n\|_{W^{1,N}(B_n)}}|^{\frac{N}{N-1}})) dx&\leq&\\
	  \Theta^{(p-1)N'}C(\alpha,N)^{N'}&<&\infty.
		\end{array}
	\end{equation}
Since
\[
u_n \rightarrow u_{\lambda}\,\, \mbox{a.e. in}\,\, B_{m+1},
\]
 by the continuity of $f$ we have
\begin{equation}\label{eq26.2zla}
f(u_n(\cdot)) \rightarrow f(u_{\lambda}(\cdot))\,\, \mbox{a.e. in}\,\, B_{m+1}.
\end{equation}
Hence, \cite[Theorem 13.44]{HS} leads to 
 \begin{equation}\label{eq16.7lza}
f(u_{n}) \to f(u_{\lambda}) \,\,\, \mbox{ weakly in } \,\,\, L^{N'}(B_{m+1}).
\end{equation}


Let us show that $u_n$ converges to $u_{\lambda}$ strongly in $W
^{1,N}(B_{m})$. To this end, fix $l\in\mathbb{N}$
and choose a smooth function $\psi_l$ satisfying $0\le \psi_l \le 1$,
$\psi_l(r)=1$ if $r\le m$ and $\psi_l(r)=0$ if $r\ge m+1/l$. Setting
$\eta_l(x):=\psi_l(|x|)$, we note that $(u_n-u_{\lambda})\eta_l\in
W_0^{1,N}(B_{m+1})\subset W_0^{1,N}(B_n)$ for any $n\ge m+1$.
Denote

$$
V_n=\int_{B_m} |\nabla u_n|^{N-2}\nabla u_n (\nabla u_n-\nabla u_{\lambda})\,dx.
$$
Using $(u_n-u_{\lambda})\eta_l$ as test function in $\eqref{z2}$ and invoking the
growth condition \eqref{growth}, we obtain
\begin{eqnarray*}
V_n
&=& \int_{|x|<m+1/l}\left(\lambda a(x)u_n^{q-1}+f(u_n)-u_n^{N-1}
\right) (u_n-u_{\lambda})\eta_l\,dx
\\
&&
   -\int_{m\le |x|<m+1/l} |\nabla u_n|^{N-2}\nabla u_n \nabla (u_n- u_{\lambda})\eta_ldx
\\
&&
   -\int_{m\le |x|<m+1/l} |\nabla u_n|^{N-2}\nabla u_n \nabla
\eta_l (u_n-u_{\lambda}) dx
\\
&\le& \int_{B_{m+1}} (\lambda a(x)u_n^{q-1}+u_n^{p-1}\phi_N(2^{\frac{N}{N-1}}\alpha |u_n|^{\frac{N}{N-1}})+u_n^{N-1})
|u_n-u_{\lambda}|\, dx
\\
& &  +\int_{m\le |x|<m+1/l}
|\nabla u_n|^{N-1}|\nabla u_{\lambda}|\,dx
\\ &&
 +d_l\int_{m\le |x|<m+1/l} |\nabla u_n|^{N-1}|u_n-u_{\lambda}|\,dx \\
&\equiv& I_n^1 +I_n^2 +I_n^3,
\end{eqnarray*}
where $d_l:=\sup_{|x|<m+1/l}|\nabla \eta_l(x)|$.

By H\"older's inequality and \eqref{eq:apriori-p} we have
\begin{eqnarray*}
I_n^1
&\le& \|u_n-u_{\lambda}\|_{L^N(B_{m+1})}
\left\{\lambda\|a\|_{L^{N/(N-q)}(B_{m+1})}\|u_n\|^q_{W^{1,N}(B_m+1)}\right.\\
&&\left. +\vartheta^{p-1}C(\alpha,N)
+\|u_n\|^{N-1}_{W^{1,N}(B_m+1)}
 \right\}
\\
&\le&  \overline{C}\|u_n-u_{\lambda}\|_{L^N(B_{m+1})},
\end{eqnarray*}
where $\overline{C}$ is a positive constant independent of $u_n$, $n$, $m$
and $l$. 

By H\"older's inequality, the following estimates
follow:
\begin{eqnarray*}
& I_n^2& \le \|u_n\|^{N-1}_{W^{1,N}(B_{m+1})}\left(\int_{m\le |x|<m+1/l}
|\nabla u_{\lambda}|^{N}\,dx\right)^{1/N},
\end{eqnarray*}

\begin{eqnarray*}
& I_n^3& \le d_l\|u_n-u_{\lambda}\|_{L^N(B_{m+1})}\|\nabla
u_n\|_{L^{N}(B_{m+1})}^{N-1}.
\end{eqnarray*}
Thereby, from \eqref{eq:apriori-p} and
\eqref{eq:main-2} we derive
\begin{eqnarray*}
\limsup_{n\to \infty} V_n \le
\tilde{\varrho}^{N-1}
\left(\int_{m\le |x|<m+1/l} |\nabla u_{\lambda}|^N\,dx\right)^{1/N},
\end{eqnarray*}
for all $l\in\mathbb{N}$. Thus, letting $l\to\infty$, we obtain
\begin{equation}\label{V1}
    \limsup_{n\to \infty} V_n\le 0.
\end{equation} As known from \eqref{eq:main-1},
$u_n$ weakly converges to $u_{\lambda}$ in $W^{1,N}(B_m)$,
so we may write
\begin{equation}\label{V2}
\begin{array}{lll}
V_n+o(1)&=&\int_{B_m}
\left(|\nabla u_n|^{N-2}\nabla u_n-|\nabla u_{\lambda}|^{N-2}\nabla u_{\lambda}\right)
(\nabla u_n-\nabla u_{\lambda})\,dx
\\
&\ge&
\left(\|\nabla u_n\|_{L^N(B_m)}^{N-1}-\|\nabla u_{\lambda}\|_{L^N(B_m)}^{N-1}\right)
\left(\|\nabla u_n\|_{L^N(B_m)}-\|\nabla u_{\lambda}\|_{L^N(B_m)}\right)
 \\
\\ & \ge& 0.
\end{array}
\end{equation}
By \eqref{V1} and \eqref{V2}, we obtain $\lim_{n\to \infty} V_n=0$, $\lim_{n\to
\infty}\|\nabla u_n\|_{L^N(B_m)}=\|\nabla u_{\lambda}\|_{L^N(B_m)}$. This implies that $u_n$ converges to $u_{\lambda}$
strongly in $W^{1,N}(B_{m})$  because the spaces
$W^{1,N}(B_{m})$ is uniformly convex.

Now, since $u_n>0$ in $B_m$, we infer that $u_{\lambda}$ is a nonnegative solution of the problem
$$
-\Delta_N u_{\lambda}+u^{N-1}_{\lambda}=\lambda a(x)u_{\lambda}^{q-1}+f(u_{\lambda}) \quad {\rm in}\ B_m,\quad u_{\lambda}\geq  0\quad
{\rm on}\ \partial  B_m.
$$

By using \eqref{utilde} and since $C_0^{\infty}(\mathbb{R}^N)$ is dense in $W^{1,N}(\mathbb{R}^N)$, by a diagonal argument, there exist a relabeled subsequence of $\{\tilde{u}_{n}\}$ and a function $u_{\lambda}\in
W^{1,N}(\mathbb{R}^N)$ such that
$$
\tilde{u}_{n} \rightharpoonup u_{\lambda} \,\,\, \mbox{in} \,\,\,
W^{1,N}(\mathbb{R}^N), 
$$
$$
\tilde{u}_n(x)\rightarrow u_{\lambda}(x)\,\,\,\mbox{for a.e } x\in \mathbb{R}^N.
$$
These convergence properties, and some iteration process, ensure that $u_{\lambda}$ is a solution of problem $\eqref{P}$, that belongs to $ C^1_{loc}(
\mathbb{R}^N)$ (see \eqref{j1}).

The next step in the proof is to show that $u_{\lambda}$ does not vanish in
$\mathbb{R}^N$. Indeed, Lemma~\ref{le2} provides a
solution $u_{\lambda,m}$ of the problem
$$\left\{
\begin{array}{lll}
-\Delta_N u+|u|^{N-2}u=\lambda a_m|u|^{q-2}u &\mbox{in}&   B_m\nonumber\\
u>0&\mbox{in}&   B_m\nonumber\\
u(x)= 0 &\mbox{on}&\partial B_m\nonumber,
\end{array} \right.
$$
where 
$$a_m=\inf_{B_m} a(x).$$
Since $\lambda a(x)t^{q-1} + f(t)\ge
\lambda a_mt^{q-1}$ for all $x\in \mathbb{R}^N$.
We are thus in a position to apply Proposition~\ref{teorsubsup} to the
functions $u_{\lambda,m}$ and $\tilde{u}_n$ with $n>m$, in place of $u_1=u_{\lambda,m}$ and
$u_2=\tilde{u}_n$, respectively, which renders $\tilde{u}_n\ge u_{\lambda,m}$ in $B_m$ for
every $n>m$. This enables us to deduce that $u_{\lambda}\geq u_{\lambda,m}$ in $B_m$, so
$u_{\lambda}(x)>0$ for  every $x\in \mathbb{R}^N$,  because $m$ was arbitrary chosen.

 Furthermore, since $\tilde{u}_n$ weakly converges to $u_{\lambda}$ in $W^{1,N}(\mathbb{R}^N)$, by \eqref{20} we have

\begin{equation}\label{rhotil}
\|u_{\lambda}\|_{W^{1,N}(\mathbb{R}^N)}\leq \liminf_{n\rightarrow\infty}\| \tilde{u}_n\|_{W^{1,N}(\mathbb{R}^N)}
= \liminf_{n\rightarrow\infty}\|\tilde{u}_n\|_{N,n}\leq \tilde{\varrho},
\end{equation}
and then, we obtain $u_{\lambda}\in
W^{1,N}(\mathbb{R}^N)$, and
\begin{equation}\label{m1}
    \|u_{\lambda}\|_{W^{1,N}(\mathbb{R}^N)} \rightarrow 0, \, \mbox{ as }\lambda\rightarrow 0. 
\end{equation}
Thus, the proof
of Theorem \ref{TP} is complete.
\subsection{Proof of Proposition \ref{prop-n}}
 In this proof, we borrowed some ideas from \cite{cgm}. 
For the sake of contradiction, suppose  that $\lambda^* = \infty$. Thus, there is a sequence $\lambda_n \to \infty$ and corresponding solutions $u_{\lambda_n} > 0$ in $\mathbb{R}^N$ given by Theorem \ref{TP}. In what follows,  define $a_R=\inf_{B_R(0)}a(r)$, for some  $R>0$ fixed.  Define also
 $$\mathcal{P}(x,t) = \lambda a(x) t^{q-1}+t^{p-1}\phi_N(\alpha|t|^{\frac{N}{N-1}})$$ and		$$
	P_1(t) = \Lambda t ^{q-1}+t^{p-1} \phi_N(\alpha|t|^{\frac{N}{N-1}}),
	$$
where $
\Lambda = \lambda \  a_R
$.	

	We claim that there exists a constant $C_{\Lambda} > 0$ such that
	\begin{equation*}
	P(x,t) \geq  C_{\Lambda} t^{N-1} \mbox{ for every }\  t>0.
	\end{equation*}
	Indeed, define the function $Q(t) = P_1(t) t^{-(N-1)}$. Then, since $q<N<p$, $Q(t) \to \infty$ as $t \to 0^+$ and as $t \to \infty$. The minimum value is $Q(t_1) = C_{\Lambda}$, where $t_1>0$ is the unique solution of
    $$
	\phi_N(\alpha t^{\frac{N}{N-1}})(p-N)t^{p-q}+\alpha \frac{N}{N-1}t^{\frac{(N-1)(p-q+1)+1}{N-1}}\phi_{N-1}(\alpha t^{\frac{N}{N-1}}) =
	\Lambda (N - q).
	$$
Indeed, let us define 
\[H(t):=\phi_N(\alpha t^{\frac{N}{N-1}})(p-N)t^{p-q}+\alpha \frac{N}{N-1}t^{\frac{(N-1)(p-q+1)+1}{N-1}}\phi_{N-1}(\alpha t^{\frac{N}{N-1}}).
\]
Note that $H(0)=0$ and $H(t) \to \infty$  as $t \to \infty$. Since 
\[
\begin{array}{rcl}
H^{\prime}(t)&=& (p-N)(p-q)t^{p-q-1}\phi_{N}(\alpha t^{\frac{N}{N-1}})\\
    &+& \frac{\alpha\,N(p-N)}{N-1}t^{p-q + \frac{1}{N}}\phi_{N-1}(\alpha t^{\frac{N}{N-1}})\\
    &+&\frac{N \alpha}{N-1}\left[\frac{(N-1)(p-q+1)+1}{N-1}\right]t^{\frac{(N-1)(p-q+1)+1}{N-1}-1}\phi_{N-1}(\alpha t^{\frac{N}{N-1}})\\
    &+&\frac{\alpha^2 N^2}{(N-1)^2}t^{\frac{(N-1)(p-q+1)+1}{N-1} + \frac{1}{N-1}}\phi_{N-2}(\alpha t^{\frac{N}{N-1}})>0, \,\,\, \forall t\geq 0,
\end{array}
\]
hence $H$ is increasing. Therefore, there is a unique point $t_1>0$ such that $H(t_1)=\Lambda(N-q)$. Consequently, by the arguments before, $t_1>0$ is the unique point of minimum of $Q$, with $Q(t_1)=C_{\Lambda}$.

Denote by $\sigma_1 > 0$ and $\varphi_1 > 0$, respectively, the principal eigenvalue and corresponding eigenfunction of the eigenvalue
problem 
$$
\left\{
\begin{array}{lll}
-\Delta_N (\varphi_1) = \sigma_1|\varphi_1|^{N-2}\varphi_1 &\mbox{in}&B_R(0) \\
\varphi_1 =0 & \mbox{on} & \partial B_R(0).
\end{array} \right.
$$
Since $C_{ \Lambda}$ increases as $\lambda_n$ increases, for $\delta > 0$ small enough, there is $\lambda_0$ in this sequence such that the corresponding constant satisfies $C_{ \Lambda_0} \geq \sigma_1 + \delta + 1$. Hence the solution $u_{\lambda_0}$ of \eqref{P} associated to $\lambda_0$ satisfies $u_{\lambda_0}>0$ in $\mathbb{R}^N$ and
$$
\left\{
\begin{array}{lll}
-\Delta_N u_{\lambda_0} \geq (C_{ \Lambda_0} - 1) |u_{\lambda_0}|^{N-2}u_{\lambda_0} \geq (\sigma_1 + \delta) |u_{\lambda_0}|^{N-2}u_{\lambda_0} &\mbox{in}&B_R(0) \\
u_{\lambda_0} \geq 0 & \mbox{on} & \partial B_R(0).
\end{array} \right.
$$
Otherwise, taking $\varepsilon>0$ small enough we obtain  $\varepsilon \varphi_1 < u_{\lambda_0}$ in $B_R(0)$ and
$$
\left\{
\begin{array}{lll}
-\Delta_N (\varepsilon \varphi_1) = \sigma_1|\varepsilon \varphi_1|^{N-2}(\varepsilon \varphi_1) \leq (\sigma_1 + \delta) |\varepsilon \varphi_1|^{N-2}(\varepsilon \varphi_1) &\mbox{in}&B_R(0) \\
\varphi_1 =0 & \mbox{on} & \partial B_R(0).
\end{array} \right.
$$
By the method of subsolution and supersolution, see \cite[Theorem 2.1]{cgm}, there is a solution $\varepsilon \varphi_1 \leq \omega \leq u_{\lambda_0}$ in $B_R(0)$ of
$$
\left\{
\begin{array}{lll}
-\Delta_N \omega = (\sigma_1 + \delta) |\omega|^{N-2}\omega &\mbox{in}&B_R(0) \\
\omega =0 & \mbox{on} & \partial B_R(0).
\end{array} \right.
$$
Hence there is a contradiction to the fact that $\sigma_1$ is isolated (see \cite{Lin}). Therefore, $\lambda^* < \infty$.

\subsection{Proof of Theorem~\ref{T:infinity}}\label{S5}


By using the same arguments as in the proof of Theorem \ref{TP}, with $D=B_n$, and considering \eqref{rhotil}, we can pass to the limit
 $n \to +\infty$ in \eqref{z7} to obtain
\begin{equation}\label{z7.2}
	   \begin{array}{c}
	\|u_{\lambda}\|_{L^\infty(\mathbb{R}^N)}
	\le C^{S_1}e^{S_2}e^{S_3}
	\max\{1,C_*\tilde{\varrho}\}.
	\end{array}
	\end{equation}
 Since $u_{\lambda}$ is bounded in $\mathbb{R}^N$, we put $M_0:=\|u_{\lambda}\|_{L^\infty(\mathbb{R}^N)}$.
Then, by Lemma~\ref{lem:element},
we have
	\begin{align}
&\int_{B(x_0,R^\prime)} f(u_{\lambda})u_{\lambda}(u_{\lambda})_M^{\beta}\eta^{N}dx\,  
\le C(\varrho)
\|u_{\lambda}\|^{\beta}_{L^{\tilde{p}(N+\beta)}(B(x_0,R^\prime))}B_{R^\prime} 
\label{eq:inf-1}\\
&\int_{B(x_0,R^\prime)} a(x) u_{\lambda}^{q}(u_{\lambda})_M^{\beta}\eta^{N}dx\,  
\le \|a\|_{L^{\gamma_1}(B(x_0,R^\prime))}
M_0^{q}\|u_{\lambda}\|_{L^{\tilde{p}(N+\beta)}(B(x_0,R^\prime))}^{\beta}
B_{R^\prime}.
\label{eq:inf-2}
\\
\intertext{By H\"{o}lder inequality, we obtain} 
&\|u_{\lambda}\|_{L^{\tilde{p}(N+\beta)}(B(x_0,R^\prime))}^{N+\beta}\le
M_0^N\|u_{\lambda}\|_{L^{\tilde{p}(N+\beta)}(B(x_0,R^\prime))}^{\beta}B_{R^\prime}
\label{eq:inf-4}.
\end{align}

Now, fix $x_0\in\mathbb{R}^N$.
It follows from the argument as in the proof of Lemma~\ref{lem:bdd}
with \eqref{eq:inf-1}, \eqref{eq:inf-2} and \eqref{eq:inf-4}
that
\begin{eqnarray}\label{eq:inf-5}
\lefteqn{\|u_{\lambda}\|_{L^{\frac{\overline{p}^*}{N}(N+\beta)}(B(x_0,R))}^{N+\beta}}\\&&
\le 2^N(N+\beta)^NC_*^NB_{R^\prime}(C_{R^\prime}+D_{R,R^\prime}) (M_0+1)^N
\|u_{\lambda}\|_{L^{\tilde{p}(N+\beta)}(B(x_0,R^\prime))}^{\beta} \nonumber
\end{eqnarray}
provided $u_{\lambda}\in L^{\tilde{p}(N+\beta)}(B(x_0,R^\prime))$.
Choose $\gamma_1$, $\tilde{p}$ and define sequences
$\{\beta_n\}$, $\{R_n^\prime\}$ and $\{R_n\}$ as in the proof of
Proposition~\ref{prop:bdd-2}.
Set
$$V_n:=\|u_{\lambda}\|_{L^{\tilde{p}(N+\beta_n)}(B(x_0,R^\prime_n))}^{\beta_n}.$$
The remainder of the proof follows with the same arguments as in in the proof of \cite[Theorem 2]{FMT} to conclude that
\begin{equation}\label{eq:thm2-00}
V_{n}^{\frac{N+\beta_{n-1}}{\beta_{n}}}
\le C (N+\beta_{n-1})^N(C_0+D_{n-1}) V_{n-1}
\end{equation}
with $C:=2^NC_*^NB_0(M_0+1)^N$.
Recall that
$$\beta_{n}+N=P^{-1}(N+\beta_{n-1}) \quad
{\rm and} \quad
\frac{N}{N+\beta_0}=P.
$$

Define
$$
Q_n:=\prod_{k=2}^{n+1} \left(1+\frac{P^k}{1-P^{k}}\right)
= \prod_{k=2}^{n+1} \left(1-P^{k}\right)^{-1}
\quad {\rm and}\quad W_n:=(C_0+D_n).
$$
Then, the inequality \eqref{eq:thm2-00} leads to
\begin{align*}
\mbox{ln} V_n &\le \frac{\beta_{n}}{N+\beta_{n-1}}
\left(\mbox{ln} V_{n-1}+\mbox{ln} [C(N+\beta_{n-1})^N]+\mbox{ln} W_{n-1}\right) \\
&=P^{-1}\left(1-P^{n+1}\right)
\left(\mbox{ln} V_{n-1}+
N\mbox{ln} [C P^{-n+1}(N+\beta_0)]+\mbox{ln} W_{n-1}\right)
\\
&\le
P^{-1}\left(1-P^{n+1}\right)\mbox{ln} V_{n-1}
+NP^{-1}\mbox{ln} [(C+1) P^{-n+1}(N+\beta_0)]+P^{-1}\mbox{ln} W_{n-1}
\\
&\le P^{-n}\left( \prod_{k=1}^n \left(1-P^{k+1}\right)
\right)
\mbox{ln} V_0 +N\sum_{k=1}^n P^{-k}\mbox{ln} [(C+1) P^{-n+k}(N+\beta_0)]
\\
&\qquad +\sum_{k=1}^n P^{-k}\mbox{ln} W_{n-k}
\\
&= P^{-n}Q_n^{-1} \mbox{ln} V_0 +N\sum_{k=1}^n P^{-k}\mbox{ln} [(C+1) P^{-n+k}(N+\beta_0)]
+\sum_{k=1}^n P^{-k}\mbox{ln} W_{n-k}
\end{align*}
for every $n$ because of 
$\mbox{ln} [(C+1) P^{-n+1}(N+\beta_0)]> 0$ and
$\mbox{ln} W_n>0$ for all $n$.
Therefore, we have
\begin{align}
&\mbox{ln} \|u_{\lambda}\|_{L^{\tilde{p}(N+\beta_n)}(B(x_0,R^\prime_n))}
=\frac{\mbox{ln} V_n}{\alpha_n}=\frac{P^n\mbox{ln} V_n}{N+\beta_0-NP^n}
\nonumber \\
&\le \frac{Q_n^{-1}\mbox{ln} V_0}{N+\beta_0-NP^n}
+\frac{\sum_{l=0}^{n-1} P^l \mbox{ln} [(C+1) P^{-l}(N+\beta_0)]}{N+\beta_0-NP^n}
+\frac{\sum_{l=0}^{n-1} P^{l}\mbox{ln} W_{l}}{N+\beta_0-NP^n}.
\label{eq:thm2-1}
\end{align}
Here, taking a sufficiently large positive constant $C^\prime$
independent of $n$, we see that
$$
\sum_{l=0}^{n-1} P^l \mbox{ln} [(C+1) P^{-l}(N+\beta_0)]
\le
C^\prime \sum_{l=0}^\infty P^l (l+1)=:S_1<\infty
$$
and
$$
\begin{array}{ll}
\sum_{l=0}^{n-1} P^{l}\mbox{ln} W_{l}&\le
C^\prime \sum_{l=0}^{n-1} P^{l} (l+1)\le C^\prime \sum_{l=0}^\infty P^{l} (l+1)
=:S_2<\infty.
\end{array}
$$
Next, we shall show that $\{Q_n\}$ is a convergent sequence. It is easily see that $\{Q_n\}$ is increasing.
Moreover,
setting $d_k:=\mbox{ln} \left(1+\frac{P^k}{1-P^{k}}\right)$,
we see that
$$
\begin{array}{ll}
\lim_{k\to\infty}\frac{d_{k+1}}{d_k}
&=
\lim_{k\to\infty}\frac{\mbox{ln} (1-P^{k+1})}{\mbox{ln} (1-P^k)}
=\lim_{k\to\infty}\frac{1-P^k}{1-P^{k+1}}\, P=P<1
\end{array}$$
by L'Hospital's rule.
This implies that
$$
\mbox{ln} Q_n=\sum_{k=2}^{n+1} \mbox{ln} \left(1+\frac{P^k}{1-P^{k}}\right)
\le \sum_{k=1}^\infty \mbox{ln} \left(1+\frac{P^k}{1-P^{k}}\right)
<\infty.
$$
Therefore, $\{Q_n\}$ is bounded from above, whence
$\{Q_n\}$ converges and
$$
1<\frac{1}{1-P^2}=Q_1\le Q_\infty:=\lim_{n\to\infty}Q_n<\infty
$$
holds.
Consequently, letting $n\to\infty$ in \eqref{eq:thm2-1},
we have
\begin{equation}\label{EST1}
\|u_{\lambda}\|_{L^\infty(B(x_0,R_*))} \le (NS_1S_2)^{\frac{1}{N+\beta_0}}
\|u_{\lambda}\|_{L^{\overline{p}^*}(B(x_0,2R_*))}^{\frac{\beta_0}{(N+\beta_0)Q_\infty}}.
\end{equation}
This yields our conclusion
since $\|u_{\lambda}\|_{L^{\overline{p}^*}(B(x_0,2R_*))}\to 0$ as $|x_0|\to \infty$,
$\beta_0>0$ and the constant $NS_1S_2$ is independent of $x_0$.

\subsection{Proof of Corollary \ref{c1}}
Notice that by \eqref{z7.2}, $\|u_{\lambda}\|_{L^\infty(\mathbb{R}^N)}$ is uniformly bounded in the variable $\lambda$. Hence, the constant $C$ in \eqref{eq:thm2-00} is uniformly bounded in the variable $\lambda$.

By \eqref{EST1}, we obtain

$$\|u_{\lambda}\|_{L^\infty(B(x_0,R_*))} \leq (NS_1S_2)^{\frac{1}{N+\beta_0}}
\|u_{\lambda}\|_{L^{\overline{p}^*}(\mathbb{R}^N)}^{\frac{\beta_0}{(N+\beta_0)Q_\infty}}\leq \tilde{C}
\|u_{\lambda}\|_{W^{1,N}(\mathbb{R}^N)}^{\frac{\beta_0}{(N+\beta_0)Q_\infty}},$$
with $\tilde{C}$ independent of $\lambda$. Since $x_0$ is arbitrary,   on has
$$\|u_{\lambda}\|_{L^\infty(\mathbb{R}^N)} \leq  \tilde{C}
\|u_{\lambda}\|_{W^{1,N}(\mathbb{R}^N)}^{\frac{\beta_0}{(N+\beta_0)Q_\infty}}.$$

Thus, \eqref{m1} ensure that 
$$\|u_{\lambda}\|_{L^{\infty}(\mathbb{R}^N)}\rightarrow 0,$$
as $\lambda \rightarrow 0$.

\subsection*{Acknowledgements}
L. F.O Faria was partially supported by FAPEMIG CEX APQ 02374/17. A.L.A. de Araujo was partially supported by FAPEMIG FORTIS-10254/2014 and CNPQ.



\end{document}